\documentclass[12pt]{article} 

\usepackage[utf8]{inputenc} 

\usepackage{geometry} 
\usepackage{graphicx} 

\usepackage{booktabs} 
\usepackage{array} 
\usepackage{paralist} 
\usepackage{verbatim} 
\usepackage{subfig} 

\usepackage{amsthm}
\usepackage{amsmath}
\usepackage{amssymb}
\usepackage{amscd}
\usepackage[all, cmtip]{xy}

\theoremstyle{plain}
\newtheorem{thm}{Theorem}[section]
\newtheorem{lem}[thm]{Lemma}
\newtheorem{prop}[thm]{Proposition}
\newtheorem{cor}[thm]{Corollary}

\theoremstyle{definition}

\newtheorem{exmp}{Example}

\theoremstyle{remark}
\newtheorem{rem}{Remark}[section]

\title{Gauss-Manin Connections for Boundary Singularities and Isochore Deformations}
\author{Konstantinos Kourliouros,\\
Imperial College London, United Kingdom}

\begin{document}

\maketitle

\begin{abstract}
We study here the relative cohomology and the Gauss-Manin connections associated to an isolated singularity of a function on a manifold with boundary, i.e. with a fixed hyperplane section. We prove several relative analogs of classical theorems obtained mainly by E. Brieskorn and B. Malgrange, concerning the properties of the Gauss-Manin connection as well as its relations with the Picard-Lefschetz monodromy and the asymptotics of integrals of holomorphic forms along the vanishing cycles. Finally, we give an application in isochore deformation theory, i.e. the deformation theory of boundary singularities with respect to a volume form.
In particular we prove the relative analog of J. Vey's isochore Morse lemma,  J. -P. Fran\c{c}oise's generalisation on the local normal forms of volume forms with respect to the boundary singularity-preserving diffeomorphisms, as well as M. D. Garay's theorem on the isochore version of Mather's versal unfolding theorem.  
\end{abstract}


\section{Introduction}

In this paper we study the Gauss-Manin connections on the relative cohomology of an isolated boundary singularity, i.e. of an isolated singularity of a function in the presence of a fixed hyperplane section, called ``the boundary'' as is usual in the literature (c.f.  \cite{A0}, \cite{A00}, \cite{A1}, \cite{A2}, \cite{Mat0}, \cite{Mat}, \cite{Szp}, \cite{Szp1} for several classification results and topological properties). Apparently, a detailed description of the Gauss-Manin connections for boundary singularities has not yet been treated, except the closely related studies \cite{Duc}, (and also \cite{Ph} and references therein) on the Gauss-Manin systems with boundary and regular analytic interactions of pairs of Lagrangian manifolds.  Here we give a generalisation, for the boundary case, of some fundamental results obtained mainly by E. Brieskorn \cite{B}, M. Sebastiani \cite{S} and B. Malgrange \cite{Mal}. More specifically we prove a relative analog of the Brieskorn-Deligne-Sebastiani theorem, concerning the finiteness and freeness of the de Rham cohomology modules and of the corresponding Brieskorn lattices associated to the boundary singularity (Theorems \ref{t1}, \ref{t3}). We also give a relative analog of the regularity theorem (Theorem \ref{t2}) according to which, the restriction of the natural Gauss-Manin connection on the localisation of the Brieskorn modules at the critical value, has regular singularities. According to the work of Brieskorn \cite{B}, the regularity of the Gauss-Manin connection, along with the algebraicity theorem and the positive solution of Hilbert's VII'th problem, give also a direct analytic proof of a relative version of the monodromy theorem (Theorem \ref{t}), i.e. that the eigenvalues of the Picard-Lefschetz monodromy operator in the relative vanishing cohomology, are indeed roots of unity. Following Malgrange \cite{Mal}, we show that the relative monodromy theorem along with the regularity theorem, give also the asymptotic expansion of the integrals of holomorphic forms along the vanishing cycles and half-cycles of the boundary singularity, when the values of the function tend to the critical one (Theorem \ref{t4}).  

These results in turn can be viewed as the first steps for the establishment of several important invariants for boundary singularities, extending those for the ordinary (i.e. without boundary) singularities, such as the spectrum, the spectral pairs and eventually, the mixed Hodge structure in the relative vanishing cohomology (c.f. \cite{Ste}, \cite{Var1}). Here we don't take this step but instead we give a direct application in isochore deformation theory, i.e. the deformation theory of boundary singularities with respect to a volume form. In particular we prove a relative analog of a J. Vey's isochore Morse lemma \cite{V}, J. -P. Fran\c{c}oise's generalisation on the local normal forms of volume forms with respect to the singularity preserving diffeomorphisms \cite{F0}, \cite{F1} (see also \cite{F2}), as well as M. D. Garay's isochore version of Mather's unfolding theorem \cite{G1}. For further possible applications of these theorems c.f. \cite{C}, \cite{G2} and references therein.

It is important to notice finally that there are two natural ways to study a boundary singularity. The first one is due to Arnol'd \cite{A1} according to which a boundary singularity can be viewed as an ordinary $\mathbb{Z}_2$-symmetric singularity after passing to the double covering space branched along the boundary (see also \cite{W} and \cite{Gor} for generalisations for other symmetric singularities). There is also another approach due to A. Szpirglas \cite{Szp}, \cite{Szp1}, according to which a boundary singularity can be viewed, at least in a (co)homological level, as an extension of two ordinary singularities, namely the ambient singularity and its restriction on the boundary. Our approach is in accordance with  the second one, i.e. we show that the relative cohomology, the relative Gauss-Manin connection and the corresponding Brieskorn lattices associated to a boundary singularity, are indeed extensions of the corresponding ordinary objects associated to the pair of isolated singularities.

\section{Relative Cohomology, Brieskorn Modules and Gauss-Manin Connections for Boundary Singularities}

We review first some basic facts concerning the topology of isolated boundary singularities.

\subsection{Milnor Numbers, (Co)homological Milnor Bundles and Topological Gauss-Manin Connections}

 Let $f:(\mathbb{C}^{n+1},0)\rightarrow (\mathbb{C},0)$ be a holomorphic function germ and let $H=\mathbb{C}^{n}\subset \mathbb{C}^{n+1}$ be a hyperplane section at the origin, which we call ``the boundary'', such that either $f$ or/and its restriction $f|_H$ on the boundary has an isolated critical point at the origin. Fix a coordinate system $(x,y_1,...,y_n)$ such that the equation of the boundary is given by $H=\{x=0\}$. The multiplicity  $\mu_{f,H}$ of the critical point, or else, the Milnor number of the boundary singularity, is the dimension of the local algebra:
\[\mathcal{Q}_{f,H}=\frac{\mathcal{O}_{n+1}}{(x\frac{\partial f}{\partial x},\frac{\partial f}{\partial y_1},...,\frac{\partial f}{\partial y_n})}, \hspace{0.3cm} \mu_{f,H}=\dim_{\mathbb{C}}\mathcal{Q}_{f,H}.\]
The Milnor number of the boundary singularity is related to the ordinary Milnor number $\mu_{f}$ of $f$:
\[\mathcal{Q}_{f}=\frac{\mathcal{O}_{n+1}}{(\frac{\partial f}{\partial x},\frac{\partial f}{\partial y_1},...,\frac{\partial f}{\partial y_n})}, \hspace{0.3cm} \mu_{f}=\dim_{\mathbb{C}}\mathcal{Q}_{f},\]
and the Milnor number $\mu_{f|_H}$ of its restriction on the boundary:
\[\mathcal{Q}_{f|_H}=\frac{\mathcal{O}_{n}}{(\frac{\partial f}{\partial y_1}|_{x=0},...,\frac{\partial f}{\partial y_n}|_{x=0})}, \hspace{0.3cm} \mu_{f|_H}=\dim_{\mathbb{C}}\mathcal{Q}_{f,H},\]
by the formula (c.f. \cite{A1}, \cite{Szp}, \cite{W}):
\[\mu_{f,H}=\mu_{f}+\mu_{f|_H}.\]

The Milnor number of a boundary singularity is an important topological invariant; let $B^{n+1}_r$ be a sufficiently small ball at the origin of $\mathbb{C}^{n+1}$ and choose a holomorphic representative $g:B^{n+1}_r\rightarrow T=g(B^{n+1}_r)$ such that its restriction $g':B^n_r\rightarrow T$ on the boundary ball $B^n_r=B^{n+1}_r\cap H$ is a holomorphic representative of the germ $f|_H$. By choosing the radius of the ball appropriately, as well as the representatives $(g,g')$, we may succeed that:
\begin{itemize}
\item the pair of fibers $(g^{-1}(0),g'^{-1}(0))$ is transversal to the pair of boundary spheres $(\partial B^{n+1}_{\epsilon}, \partial B^{n}_{\epsilon})$ for all $\epsilon < r$, and it has an isolated singularity at the origin (the fiber $g^{-1}(0)$ might be smooth but not transversal to the hyperplane $H$),
\item the pair of fibers $(g^{-1}(t),g'^{-1}(t))$ is smooth and transversal to the boundary spheres $(\partial B^{n+1}_{\epsilon}, \partial B^{n}_{\epsilon})$ for some $\epsilon$  over all points $t\in \bar{S}$ of the closure of a sufficiently small open disc $S\subset T$ centered at the origin. 
\end{itemize}
The standard representative $f:X\rightarrow S$ is obtained by restricting $g$ to $X=\mathring{B}^{n+1}_{\epsilon}\cap g^{-1}(S)$ and is such that its restriction $f':X'=X\cap H\rightarrow S$ is a standard representative of $f|_H$ in the sense that it is obtained by the restriction of $g'$ on $X'=\mathring{B}^n_{\epsilon}\cap g'^{-1}(S)$. Thus one obtains a diagram of standard representatives:
\begin{center}.
\xymatrix{
X  \ar[r]^f & S\\
X' \ar@{^{(}->}[u]^i \ar[ru]_{f'}}.
\end{center}
which we denote by $(f,f'):(X,X')\rightarrow S$. We will call it the standard (or Milnor) representative of the boundary singularity $(f,H)$.

Denote now by $(X_0=f^{-1}(0),X'_0=f'^{-1}(0))$ the pair of singular fibers and let $(X^*=X\setminus X_0,X'^*=X'\setminus X'_0)$ be their corresponding complements. Then for $S^*=S\setminus 0$, the restriction of $(f,f')$ on $(X^*,X'^*)$ induces a $C^{\infty}$-fiber bundle pair (by Ehresmann's fibration theorem), i.e. a diagram of $C^{\infty}$-fiber bundles:
\begin{center}.
\xymatrix{
X^*  \ar[r]^f & S^*\\
X'^* \ar@{^{(}->}[u]^i \ar[ru]_{f'}},
\end{center}
which we denote again by $(f,f'):(X^*,X'^*)\rightarrow S^*$. Let $(X_t=f^{-1}(t),X'_t=f'^{-1}(t))$ be a pair of regular fibers. In particular the fiber $X_t$ is smooth and transversal to the boundary $X'$, so that its intersection $X'_t$ with the boundary is a smooth submanifold of both $X'$ and $X_t$. According to a theorem of Arnol'd \cite{A2} which generalises the Milnor-Palamodov theorem \cite{Mil}, \cite{Pal} for the boundary case, the manifold $X_t/X'_t$ has the homotopy type of a bouquet of $\mu_{f,H}$ $n$-dimensional spheres, where $\mu_{f,H}=\dim_{\mathbb{C}}\mathcal{Q}_{f,H}$ is the Milnor number of the boundary singularity $(f,H)$. In particular, $\mu_{f,H}$ is exactly equal to the rank of the relative homology group $H_n(X_t,X'_t)$ (it can be considered with integer coefficients). The equality $\mu_{f,H}=\mu_{f}+\mu_{f|_H}$ follows then from the long exact sequence in homology induced by the embedding $i_t: X'_t\hookrightarrow X_t$ and the Milnor-Palamodov theorem for the pair $(f,f')$ respectively, according to which: 
\[H_n(X_t)\cong \mathbb{Z}^{\mu_{f}}, \hspace{0.3cm} H_{n-1}(X'_t)\cong \mathbb{Z}^{\mu_{f|_H}}\]
(all other homologies of $X_t$ and $X'_t$ are zero, except in zero degree).  Indeed, the long exact homology sequence reduces to the short exact sequence: 
\[0 \rightarrow H_n(X_t)\rightarrow H_n(X_t,X'_t)\stackrel{\partial}{\rightarrow} H_{n-1}(X'_t)\rightarrow 0\]
and thus
\[H_n(X_t,X'_t)\cong \mathbb{Z}^{\mu_{f}+\mu_{f|_H}}.\]
A basis of the relative homology group $H_n(X_t,X'_t)$ is obtained by the $\mu_{f}$ ordinary vanishing cycles of $f$  and the $\mu_{f|_H}$ vanishing half-cycles, i.e. those relative cycles of $X_t$ which cover  the $\mu_{f|_H}$ ordinary vanishing cycles of $f|_H$  inside $X_t\setminus X'_t$  (c.f. \cite{A2}, \cite{Szp}). 

By obvious duality, to the short exact homology sequence above there corresponds a short exact sequence in cohomology:
\begin{equation}
\label{ses0}
0\rightarrow H^{n-1}(X'_t)\stackrel{\delta}{\rightarrow}H^n(X_t,X'_t)\rightarrow H^n(X_t)\rightarrow 0,
\end{equation}
with the standard formal adjoint formula for the boundary and coboundary operators $(\partial,\delta)$:
\[<\delta \alpha, \gamma >=<\alpha, \partial \gamma>,\]
where $<.,.>$ is the natural duality morphism between relative homology and  cohomology:
\[<.,.>:H^n(X_t,X'_t)\times H_n(X_t,X'_t)\rightarrow \mathbb{Z}.\]

 In order to study the variations in cohomology of the Milnor fibers as $t$ varies in $S^*$  it is convenient to consider the cohomologies above as with complex coefficients, and endowed with their canonical integral lattices. Since the pair $(f,f'):(X^*,X'^*)\rightarrow S^*$ is a $C^{\infty}$-fiber bundle pair over the 1-dimensional manifold $S^*$, the vector spaces $H^p(X_t;\mathbb{C})$, $H^p(X'_t;\mathbb{C})$ and $H^p(X_t,X'_t;\mathbb{C)}$,  glue together to form the fibers of the corresponding cohomological (or Milnor) vector bundles:
\[\bigcup_{t\in S^*}H^{p}(X_t;\mathbb{C})\rightarrow S^*,\] 
\[\bigcup_{t\in S^*}H^{p}(X'_t;\mathbb{C})\rightarrow S^*,\]
\[\bigcup_{t\in S^*}H^{p}(X_t,X'_t;\mathbb{C})\rightarrow S^*.\]
The transition functions in each of these bundles are locally constant (because of integrality) and thus the vector bundles above are holomorphic flat vector bundles, each endowed with its own topological Gauss-Manin connection, defined by the condition that the horizontal sections are generated by the corresponding local systems $R^pf_*\mathbb{C}_{X^*}$, $R^pf_*\mathbb{C}_{X'^*}$ and $R^pf_*\mathbb{C}_{X^*\setminus X'^*}$, where the sheaves $\mathbb{C}_{X'^*}$, $\mathbb{C}_{X^*\setminus X'^*}$ are the extensions by zero of the restrictions of the constant sheaf $\mathbb{C}_{X^*}$ on the closed subspace $X'^*$ and its open complement $X^*\setminus X'^*$ respectively.  In particular, if we consider the sheaves of sections of each of the cohomological fibrations:
\[\mathcal{H}^p(X^*/S^*)=R^pf_*\mathbb{C}_{X^*}\otimes_{\mathbb{C}_{S^*}}\mathcal{O}_{S^*}\]
\[\mathcal{H}^p(X'^*/S^*)=R^pf_*\mathbb{C}_{X'^*}\otimes_{\mathbb{C}_{S^*}}\mathcal{O}_{S^*}\]
and
\[\mathcal{H}^p(X^*,X'^*/S^*)=R^pf_*\mathbb{C}_{X^*\setminus X'^*}\otimes_{\mathbb{C}_{S^*}}\mathcal{O}_{S^*},\]
then, the (topological) Gauss-Manin connections are defined by the conditions:
\[R^pf_*\mathbb{C}_{X^*}=\ker D_f, \hspace{0.2cm}R^pf_*\mathbb{C}_{X'^*}=\ker D_{f|_H},\]
\[R^pf_*\mathbb{C}_{X^*\setminus X'^*}=\ker D_{f,H},\] 
where
\[ D_f:\mathcal{H}^p(X^*/S^*)\rightarrow \mathcal{H}^p(X^*/S^*), \hspace{0.2cm} D_{f|_H}:\mathcal{H}^p(X'^*/S^*)\rightarrow \mathcal{H}^p(X'^*/S^*),\]
and
\[D_{f,H}:\mathcal{H}^p(X^*,X'^*/S^*)\rightarrow \mathcal{H}^p(X^*,X'^*/S^*),\]
are the covariant derivatives of the corresponding connections. Each one of these connections is determined in turn by differentiating locally constant sections of the corresponding cohomology bundle along the vector field $d/dt$ on the base $S^*$ (where $f=t$ is a local coordinate) by the rule:
\[D(c\otimes g)=c\otimes \frac{dg}{dt}, \]  
where $c$ is a section of the corresponding local system and $g$ is a holomorphic function of $t$. We will call the two Gauss-Manin connections $D_f$ and $D_{f|_H}$ ordinary, and the Gauss-Manin connection $D_{f,H}$ relative.

The cohomological Milnor bundles and the Gauss-Manin connections above are not independent with each other but they are connected through long exact sequences; first there is a long exact sequence of local systems:
\[...\rightarrow R^{p-1}f_*\mathbb{C}_{X'^*}\rightarrow R^pf_*\mathbb{C}_{X^*\setminus X'^*}\rightarrow  R^{p}f_*\mathbb{C}_{X^*}\rightarrow  R^{p}f_*\mathbb{C}_{X'^*}\rightarrow ...,\]
obtained by applying the direct image functor $Rf_*$ to the short exact sequence of constant sheaves:
\[0\rightarrow \mathbb{C}_{X^*\setminus X'^*}\rightarrow \mathbb{C}_{X^*}\rightarrow \mathbb{C}_{X'^*}\rightarrow 0.\] 
There is also a long exact sequence of sheaves of sections of the cohomology bundles:
\begin{equation}
\label{les0}
...\rightarrow \mathcal{H}^{p-1}(X'^*/S^*) \rightarrow \mathcal{H}^{p}(X^*,X'^*/S^*) \rightarrow \mathcal{H}^{p}(X^*/S^*) \rightarrow \mathcal{H}^{p}(X'^*/S^*) \rightarrow ...
\end{equation}
obtained by the long exact sequence of local systems above after tensoring with $\otimes_{\mathbb{C}_{S^*}}\mathcal{O}_{S^*}$. In particular, the long exact sequence of the cohomology sheaves  is a long exact sequence of locally free sheaves of coherent $\mathcal{O}_{S^*}$-modules which, according to Milnor's (or Arnol'd's) theorem reduces to the short exact sequence:
\begin{equation}
\label{ses1}
0\rightarrow \mathcal{H}^{n-1}(X'^*/S^*)\rightarrow \mathcal{H}^n(X^*,X'^*/S^*)\rightarrow \mathcal{H}^n(X^*/S^*)\rightarrow 0.
\end{equation}
It follows  that the relative cohomology sheaf $\mathcal{H}^n(X^*,X'^*/S^*)$ is an extension of the sheaf $\mathcal{H}^{n-1}(X'^*/S^*)$ by $\mathcal{H}^n(X^*/S^*)$ and the relative Gauss-Manin connection $D_{f,H}$ on it is an extension of the two ordinary Gauss-Manin connections $D_{f|_H}$, $D_{f}$. In particular the restriction of the relative Gauss-Manin connection $D_{f,H}$ on the sheaf $\mathcal{H}^{n-1}(X'^*/S^*)$ can be identified with the ordinary Gauss-Manin connection $D_{f|_H}$ while the quotient connection induced on $\mathcal{H}^n(X^*/S^*)$ can be identified with the ordinary Gauss-Manin connection $D_f$.

 On the other hand, it is well known (c.f. \cite{Del}) that any local system on $S^*$ with a flat connection is determined by the monodromy, i.e. the representation of the fundamental group $\pi_1(S^*,t)$ on its fibers, and conversely, the monodromy determines the connection.   Here we may choose the standard representatives $(f,f')$ in such a way so that the geometric monodromy on the fibers $X_t$ induced by travelling once around the origin in the positive direction, leaves the subfiber $X'_t$ invariant. Thus we obtain representations of the fundamental group $\pi_1(S^*,t)=\mathbb{Z}$ in the group of automorphisms of the fibers of the corresponding cohomological bundles. Let $T_{f|_H}\in \text{Aut}H^{n-1}(X'_t;\mathbb{C})$, $T_{f}\in \text{Aut}H^{n}(X_t;\mathbb{C})$ be the ordinary linear transformations in cohomology, i.e. the well known  Picard-Lefschetz monodromy transformations, and denote by $T_{f,H}\in \text{Aut}H^n(X_t,X'_t;\mathbb{C})$ the linear transformation induced in relative cohomology. We will call this transformation the relative Picard-Lefschetz monodromy (as in \cite{Szp}). By the above, it is an extension of the two ordinary Picard-Lefschetz monodromies, i.e. there is a commutative diagram:
\begin{equation}
\begin{CD}
0 @>>>H^{n-1}(X'_t;\mathbb{C}) @>\delta >> H^n(X_t,X'_t;\mathbb{C}) @>p>> H^n(X_t;\mathbb{C}) @>>> 0 \\
@.             @ V T_{f|_H} VV                           @ V T_{f,H} VV                 @ VT_f VV          @.          \\
0 @>>>H^{n-1}(X'_t;\mathbb{C}) @>\delta >> H^n(X_t,X'_t;\mathbb{C}) @>p>> H^n(X_t;\mathbb{C}) @>>> 0\\
\end{CD}
\end{equation}
By the fact that both $T_{f|_H}$ and $T_f$ are isomorphisms it follows that $T_{f,H}$ is also an isomorphism. Concerning its eigenvalues we have the following relative analog of the monodromy theorem:
\begin{thm}
\label{t}
The eigenvalues of the relative monodromy operator $T_{f,H}$ are roots of unity. 
\end{thm}
The proof follows immediately by the fact that the characteristic polynomial of $T_{f,H}$ is the product of the characteristic polynomials of $T_{f|_H}$ and $T_{f}$, whose roots are, by the ordinary monodromy theorem  (c.f. Brieskorn \cite{B}), roots of unity. Another straightforward analytic proof of the relative monodromy theorem may be derived, following Brieskorn, by the results of the next sections (see Remark \ref{r2}).
\begin{rem}
The statement of the theorem above is, as is usually called, the first part of the monodromy theorem. The second part, concerning the bound on the maximal size of the Jordan blocks, is more complicated and it will not be discussed here. Possibly, a sharper bound than the obvious one $\leq n-1+n=2n-1$, may be obtained either using resolution of singularities and a Clemens construction as in \cite{Cl}, or using the eventual mixed Hodge structure on the vanishing relative cohomology $H^n(X_t,X'_t;\mathbb{C})$ (as for example in  \cite{Ste}, \cite{Var1}).    
\end{rem}

\subsection{Relative de Rham Cohomology, Analytic Gauss-Manin Connections and Brieskorn Modules}

Since the pair of Milnor fibers $(X_t,X'_t)$ is Stein, its cohomologies can be computed using holomorphic differential forms and the corresponding relative de Rham cohomologies.

\subsubsection{The Brieskorn-Deligne Theorem for Boundary Singularities}

Recall that for a single morphism $f:X\rightarrow S$ the complex of holomorphic relative differential forms $\Omega^{\bullet}_{X/S}$ is defined as the quotient complex (c.f. \cite{Gro}):
\[\Omega^{\bullet}_{X/S}=\frac{\Omega^{\bullet}_X}{df\wedge \Omega^{\bullet-1}_X},\]
where $\Omega^{\bullet}_X$ is the complex of holomorphic forms on $X$ and $f^*\Omega^1_{S}=df$ is the ideal sheaf generated by the differential of $f$. The differential $d$ (called the relative differential and denoted also by $d_{X/S}$) of the relative de Rham complex $\Omega^{\bullet}_{X/S}$ is the one induced by the absolute differential $d_X$ of the complex $\Omega^{\bullet}_X$ and it is $f^{-1}\mathcal{O}_S$-linear.  For a pair of standard representatives $(f,f'):(X,X')\rightarrow S$, one may define several other relative de Rham complexes, with the most obvious one being the relative de Rham complex $\Omega^{\bullet}_{X'/S}$ of the map $f':X'\rightarrow S$, viewed independently of the embedding $i:X'\hookrightarrow X$. Indeed, we have as above:
\[\Omega^{\bullet}_{X'/S}:=\frac{\Omega^{\bullet}_{X'}}{df'\wedge \Omega^{\bullet-1}_{X'}},\]
where the relative differential $d_{X'/S}$ is induced by the differential $d_{X'}$ and it is also $f'^{-1}\mathcal{O}_{S}$-linear. Consider now its extension by zero $i_*\Omega^{\bullet}_{X'/S}$ in $X$. Since $X'$ is closed and smooth we have an epimorphism of analytic modules, which is the restriction morphism induced by the pullback map:
\[i^*:\Omega^{\bullet}_{X/S}\rightarrow i_*\Omega^{\bullet}_{X'/S}.\]
The kernel of this morphism is the subcomplex $\Omega^{\bullet}_{X/S}(X')\subset \Omega^{\bullet}_{X/S}$ consisting of relative differential forms whose support lies in the complement $X\setminus X'$ and in particular they vanish when restricted to the hypersurface $X'$. More specifically, let $\Omega^{\bullet}_{X}(X')\subset \Omega^{\bullet}_X$ be the subcomplex of holomorphic forms on $X$ which vanish when restricted on $X'$. This fits in a short exact sequence of complexes:
\[0\rightarrow \Omega^{\bullet}_{X}(X')\rightarrow \Omega^{\bullet}_X\rightarrow i_*\Omega^{\bullet}_{X'}\rightarrow 0,\]
from which we obtain the obvious isomorphism:
\[i_*\Omega^{\bullet}_{X'}\cong \frac{\Omega^{\bullet}_X}{\Omega^{\bullet}_X(X')}\]
(notice that by definition, the complex of holomorphic forms on $X'$ can be identified with the restriction on $X'$ of the above quotient complex).
Consider now muliplication with $df\wedge $ in the short exact sequence above. It gives a commutative diagram:
\begin{equation}
\label{cd3}
\begin{CD}
@. 0   @. 0 @. 0 \\
@. @VVV           @VVV               @VVV               \\
0 @>>>df\wedge \Omega^{\bullet-1}_{X}(X') @>>>df\wedge \Omega^{\bullet-1}_{X} @>>> i_*(df'\wedge \Omega^{\bullet-1}_{X'}) @>>> 0 \\
@. @VVV           @VVV               @VVV               \\
0 @>>>\Omega^{\bullet}_{X}(X') @>>> \Omega^{\bullet}_{X} @>>> i_*\Omega^{\bullet}_{X'} @>>> 0\\
@. @VVV           @VVV               @VVV               \\
0 @>>> \Omega^{\bullet}_{X/S}(X') @>>>  \Omega^{\bullet}_{X/S} @>>>  i_*\Omega^{\bullet}_{X'/S} @>>> 0\\
@. @VVV           @VVV               @VVV               \\
@. 0   @. 0 @. 0 \\
\end{CD}
\end{equation}
where the last row consists of the relative de Rham complexes: 
\[\Omega^{\bullet}_{X/S}(X'):=\frac{\Omega^{\bullet}_{X}(X')}{df\wedge \Omega^{\bullet-1}_{X}(X')}, \hspace{0.3cm}\Omega^{\bullet}_{X/S}=\frac{\Omega^{\bullet}_{X}}{df\wedge \Omega^{\bullet-1}_X},\]
\[i_*\Omega^{\bullet}_{X'/S}:=i_*\frac{\Omega^{\bullet}_{X'}}{df'\wedge \Omega^{\bullet-1}_{X'}}.\]
By the fact that all the columns  and the first two rows in the above diagram are exact, it follows from the 9-lemma that the lower sequence of relative de Rham complexes is also exact and thus there is an isomorphism:
\[i_*\Omega^{\bullet}_{X'/S}\cong \frac{\Omega^{\bullet}_{X/S}}{\Omega^{\bullet}_{X/S}(X')},\]
which implies that  the complex $\Omega^{\bullet}_{X/S}(X')$ can be indeed identified with the kernel of the restriction morphism $i^*:\Omega^{\bullet}_{X/S}\rightarrow i_*\Omega^{\bullet}_{X'/S}$. 

Recall now that if $\mathcal{F}^{\bullet}$ is a complex of analytic sheaves on $X$ with an $f^{-1}\mathcal{O}_S$-linear differential, then its cohomology sheaves are defined by the hyperdirect image sheaves $\mathbb{R}^pf_*\mathcal{F}^{\bullet}$, which are defined in turn by the hypercohomology presheaves: 
\[S\supset U\mapsto\mathbb{H}^p(f^{-1}(U), \mathcal{F}^{\bullet}).\] 
Moreover, for a Stein morphism, it follows from Cartan theorems that these do indeed compute the cohomology $H^p(\mathcal{F}^{\bullet})|_{f^{-1}(U)}$. Recall also that if $\mathcal{F}^{\bullet}$ is a complex of analytic sheaves defined on the closed smooth subspace $X'$ with an $f'^{-1}\mathcal{O}_{S}$-linear differential then, if we denote by $i_*\mathcal{F}^{\bullet}$ its extension by zero on $X$, we have a natural isomorphism of cohomology sheaves:
\begin{equation}
\label{gro}
\mathbb{R}^pf_*i_*\mathcal{F}^{\bullet}\cong \mathbb{R}^pf'_*\mathcal{F}^{\bullet}.
\end{equation}
Indeed, this follows from the Groethendieck spectral sequence for the composition $f'=f\circ i$ and the fact that the direct image $i_*$ of a closed embedding is exact (i.e. its higher direct images are all zero).  

Now, if $\mathcal{F}^{\bullet}$ is one of the above complexes of relative forms then we write the relative de Rham cohomology sheaves as:
 \[\mathcal{H}^p_{dR}(X,X'/S)=\mathbb{R}^pf_*\Omega^{\bullet}_{X/S}(X'), \hspace{0.2cm} \mathcal{H}^p_{dR}(X/S)=\mathbb{R}^pf_*\Omega^{\bullet}_{X/S},\]
\[\mathcal{H}^p_{dR}(X'/S)=\mathbb{R}^pf_*i_*\Omega^{\bullet}_{X'/S}\cong \mathbb{R}^pf'_*\Omega^{\bullet}_{X'/S}\] 
respectively, where the last isomorphism follows from the isomorphism  (\ref{gro}) above. The short exact sequence:
\begin{equation}
\label{ses1}
0\rightarrow \Omega^{\bullet}_{X/S}(X')\rightarrow \Omega^{\bullet}_{X/S}\rightarrow i_*\Omega^{\bullet}_{X'/S}\rightarrow 0
\end{equation}
gives, after application of the hyperdirect image functor $\mathbb{R}f_*$, a long exact sequence in cohomology:
\begin{equation}
\label{les1}
...\rightarrow \mathcal{H}_{dR}^{p-1}(X'/S)\stackrel{\delta}{\rightarrow} \mathcal{H}_{dR}^{p}(X,X'/S)\rightarrow \mathcal{H}_{dR}^{p}(X/S)\rightarrow \mathcal{H}_{dR}^{p}(X'/S) \rightarrow ...
\end{equation} 
which possesses the following important properties summarised in the following relative analog of the Brieskorn-Deligne-Sebastiani theorem:
\begin{thm}
\label{t1}

\noindent 
\begin{itemize}
\item[(i.)] The long exact sequence (\ref{les1}) is a long exact sequence of coherent sheaves of locally free $\mathcal{O}_S$-modules. 
\item[(ii.)] It is isomorphic over $S^*$ with the long exact sequence (\ref{les0}) of sheaves of sections of the corresponding cohomological Milnor bundles.
\item[(iii.)] The stalk at the origin of the long exact sequence (\ref{les1}) is isomorphic to the long exact sequence of free $\mathcal{O}_{S,0}$-modules of finite type:
\begin{equation}
\label{les2}
\rightarrow H^{p-1}(\Omega^{\bullet}_{X'/S,0})\stackrel{\delta}{\rightarrow}H^p(\Omega^{\bullet}_{X/S,0}(X',0))\rightarrow H^p(\Omega^{\bullet}_{X/S,0})\rightarrow H^p(\Omega^{\bullet}_{X'/S,0})\rightarrow
\end{equation}
which is the long exact cohomology sequence induced from the stack at the origin of the short exact sequence (\ref{ses1}).
\end{itemize} 
\end{thm}  
\begin{proof}

\noindent (i.) (iii.)  Since the singularities are isolated the proof of coherence in (i) as well as the isomorphism at the origin with the long exact sequence (\ref{les2}) in (iii), follows immediately from Kiehl-Verdier type theorems related to the relative constructibility of these sheaves (c.f. \cite{G0}). Alternatively, we know from the ordinary Brieskorn-Deligne theorem that the sheaves $\mathbb{R}^pf_*\Omega^{\bullet}_{X/S}$ and $\mathbb{R}^pf'_*\Omega^{\bullet}_{X'/S}$ are already coherent, from which it follows (by the long exact sequence (\ref{les1})) that the sheaves $\mathbb{R}^pf_*\Omega^{\bullet}_{X/S}(X')$ are coherent as well.  The property (iii) also holds for $\mathbb{R}^pf_*\Omega^{\bullet}_{X/S}(X')$ because it holds for the other two sheaves; indeed if $X_0=f^{-1}(0)$ is the singular fiber, one has a commutative diagram of canonical restriction morphisms:
\[\begin{CD}
\label{cd4}
0 @>>>\Gamma(X_0,\Omega^{\bullet}_{X/S}(X')) @>>> \Gamma(X_0,\Omega^{\bullet}_{X/S}) @>>> \Gamma(X_0,i_*\Omega^{\bullet}_{X'/S}) @>>> 0\\
@. @VVV           @VVV               @VVV               \\
0 @>>> \Omega^{\bullet}_{X/S,0}(X',0) @>>>  \Omega^{\bullet}_{X/S,0} @>>>  i_*\Omega^{\bullet}_{X'/S,0} @>>> 0\\
\end{CD}\]
where the middle and right morphisms are quasi-isomorphisms. It follows by the 5-lemma that the left morphism is a quasi-isomorphism as well.  Thus, it suffices to show that the sheaves  are locally free. But for $p<n$ all the sheaves in (\ref{les1}) are endowed with Gauss-Manin connections which makes them locally free. Indeed, for the sheaves $\mathbb{R}^{p}f_*\Omega^{\bullet}_{X/S}$ and $\mathbb{R}^{p-1}f'_*\Omega^{\bullet}_{X'/S}$ this was proved by Brieskorn, whereas for $\mathbb{R}^{p}f_*\Omega^{\bullet}_{X/S}(X')$ it will be shown in the next section. For $p=n$ it follows from Milnor's (or Arnol'd's) theorem that there is a short exact sequence of coherent sheaves:
\[0\rightarrow \mathbb{R}^{n-1}f'_*\Omega^{\bullet}_{X'/S}\rightarrow \mathbb{R}^{n}f_*\Omega^{\bullet}_{X/S}(X')\rightarrow \mathbb{R}^{n}f_*\Omega^{\bullet}_{X/S}\rightarrow 0\]
By the Sebastiani theorem \cite{S} the sheaves on the left and on the right are locally free and it follows that the middle one is also locally free. 

\noindent (ii.)  This property is also classical and it guarantees that the de Rham cohomology sheaves are indeed coherent extensions of the sheaves of sections of the corresponding cohomological bundles at the origin. Briefly, one uses the relative Poincar\'e lemma according to which over the smooth points $S^*$, the short exact sequence:
\[0\rightarrow f^{-1}\mathcal{O}_{S^*}|_{X^*\setminus X'^*}\rightarrow f^{-1}\mathcal{O}_{S^*}\rightarrow f^{-1}\mathcal{O}_{S^*}|_{X'^*}\rightarrow 0,\]
where the left and right terms are the extension by zero of the restriction of the sheaf $f^{-1}\mathcal{O}_{S^*}$ on $X^*\setminus X'^*$ and $X'^*$ respectively, is a resolution of the short exact sequence (\ref{ses1}), i.e. there is a commutative diagram:
\[\begin{CD}
\label{cd5}
@. 0   @. 0 @. 0 \\
@. @VVV           @VVV               @VVV               \\
0 @>>>f^{-1}\mathcal{O}_{S^*}|_{X^*\setminus X'^*} @>>> f^{-1}\mathcal{O}_{S^*} @>>> f^{-1}\mathcal{O}_{S^*}|_{X'} @>>> 0\\
@. @VVV           @VVV               @VVV               \\
0 @>>> \Omega^{\bullet}_{X^*/S^*}(X'^*) @>>>  \Omega^{\bullet}_{X^*/S^*} @>>>  i_*\Omega^{\bullet}_{X'^*/S^*} @>>> 0\\
\end{CD}\]
From this, one obtains the required isomorphisms (c.f. \cite{B}, \cite{Lo}):
\[\mathbb{R}^pf_*\Omega^{\bullet}_{X^*/S^*} \cong R^pf_*f^{-1}\mathcal{O}_{S^*}\cong R^pf_*\mathbb{C}_{X^*}\otimes_{\mathbb{C}_{S^*}}\mathcal{O}_{S^*},\]
\[\mathbb{R}^pf_*\Omega^{\bullet}_{X'^*/S^*} \cong R^pf_*(f^{-1}\mathcal{O}_{S^*}|_{X'^*})\cong R^pf_*\mathbb{C}_{X'^*}\otimes_{\mathbb{C}_{S^*}}\mathcal{O}_{S^*},\]
and finally:
\[\mathbb{R}^pf_*\Omega^{\bullet}_{X^*/S^*}(X'^*) \cong R^pf_*(f^{-1}\mathcal{O}_{S^*}|_{X^*\setminus X'^*})\cong R^pf_*\mathbb{C}_{X^*\setminus X'^*}\otimes_{\mathbb{C}_{S^*}}\mathcal{O}_{S^*}.\]

\end{proof}

In the theorem above, property (iii) is of great significance in the sense that the long exact sequence (\ref{les2}) is an invariant of the boundary singularity germ $(f,H)$, i.e. it does not depend on all other choices (e.g. the standard representatives). For convenience in the following let us change notation for the relative de Rham complexes associated to the the germ $(f,H)$:
\[\Omega^{\bullet}_{X/S,0}:=\Omega^{\bullet}_f=\frac{\Omega^{\bullet}}{df\wedge \Omega^{\bullet-1}}, \hspace{0.3cm} \Omega^{\bullet}_{X/S,0}(X',0):=\Omega^{\bullet}_{f}(H)=\frac{\Omega^{\bullet}(H)}{df\wedge \Omega^{\bullet-1}(H)},\]
\[i_*\Omega^{\bullet}_{X'/S,0}:=i_*\Omega^{\bullet}_{f|_H}= \frac{i_*\Omega^{\bullet}_H}{i_*(df'\wedge \Omega^{\bullet-1}_H)},\]
where $\Omega^{\bullet}:=\Omega^{\bullet}_{X,0}$ is the complex of germs of holomorphic forms at the origin of $\mathbb{C}^{n+1}$, $\Omega^{\bullet}(H)=x\Omega^{\bullet}+dx\wedge \Omega^{\bullet-1}\subset \Omega^{\bullet}$ is the subcomplex of forms vanishing on $H$ and 
\[i_*\Omega^{\bullet}_H\cong \frac{\Omega^{\bullet}}{\Omega^{\bullet}(H)}=\frac{\Omega^{\bullet}}{x\Omega^{\bullet}+dx\wedge \Omega^{\bullet-1}}\] 
is the quotient complex (the extension by zero of the complex of sheaves of germs of holomorphic forms defined on $H=\mathbb{C}^n\subset \mathbb{C}^{n+1}$) The stack at the origin of the short exact sequence (\ref{ses1}) is written now:
\[0\rightarrow \Omega^{\bullet}_{f}(H)\rightarrow \Omega^{\bullet}_f\rightarrow i_*\Omega^{\bullet}_{f|_H}\rightarrow 0,\]
whereas the induced long exact cohomology sequence (\ref{les2}) is written:
\begin{equation}
\label{les3}
...\rightarrow H^{p-1}(\Omega^{\bullet}_{f|_H})\stackrel{\delta}{\rightarrow}H^p(\Omega^{\bullet}_{f}(H))\rightarrow H^p(\Omega^{\bullet}_f)\rightarrow H^{p}(\Omega^{\bullet}_{f|_H})\rightarrow ...
\end{equation}
and it is a long exact sequence of free $\mathbb{C}\{f\}$-modules of finite type. In particular, the long exact sequence (\ref{les3}) above reduces to the short exact sequence:
\begin{equation}
\label{ses2}
0\rightarrow H^{n-1}(\Omega^{\bullet}_{f|_H})\stackrel{\delta}{\rightarrow}H^n(\Omega^{\bullet}_{f}(H))\rightarrow H^n(\Omega^{\bullet}_f)\rightarrow 0.
\end{equation}
The connecting morphism $\delta$ is defined as follows: let $\bar{\alpha} \in \Omega^{n-1}_{f}$ represent a class $\alpha \in  H^{n-1}(\Omega^{\bullet}_{f|_H})=H^{n-1}(\frac{\Omega^{\bullet}_f}{\Omega^{\bullet}_{f}(H)})$. Then $d\bar{\alpha} \in \Omega^{n}_{f}(H)$ is closed and defines a class $d\bar{\alpha} \in H^n(\Omega^{\bullet}_{f}(H))$. By definition $\delta \alpha=d\bar{\alpha}$. Obviously this map is $\mathbb{C}\{f\}$-linear and it is independent of the representatives, but depends only on the class $\alpha$.

As a corollary we obtain:
\begin{cor}
\[H^p(\Omega^{\bullet}_{f|_H})\cong  \left\{\begin{array}{cl}
								 \mathbb{C}\{f\}, & p=0 \\
								 0, & 0< p<n-1, \\
								 \mathbb{C}\{f\}^{\mu_{f|_H}}, & p=n-1,
								 \end{array} \right. \hspace{0.2cm} H^p(\Omega^{\bullet}_{f})\cong  \left\{\begin{array}{cl}
								 \mathbb{C}\{f\}, & p=0 \\
								 0, & 0< p<n, \\
								 \mathbb{C}\{f\}^{\mu_{f}}, & p=n,
								 \end{array} \right.\]
\[H^p(\Omega^{\bullet}_{f}(H))\cong  \left\{\begin{array}{cl}
								 0, & 0\leq p<n, \\
								 \mathbb{C}\{f\}^{\mu_{f,H}}, & p=n,
								 \end{array} \right.\]
where $\mu_{f,H}=\mu_{f|_H}+\mu_{f}$ is the Milnor number of the boundary singularity $(f,H)$.
\end{cor}

\subsubsection{The Relative Gauss-Manin Connection and Relative Brieskorn Modules}

Here we will define first the analytic relative Gauss-Manin connection $D_{f,H}$ on the de Rham cohomology sheaves $\mathcal{H}^p_{dR}(X,X'/S)$ and we will show that it coincides with the topological one defined on the cohomology sheaves $\mathcal{H}^p(X^*,X'^*/S^*)$. This will imply also that the de Rham cohomology sheaves are indeed locally free and will finish the proof of Theorem \ref{t1}, (iv). To start let us make explicit the isomorphism:
\begin{equation}
\label{dr0}
\mathcal{H}^p_{dR}(X^*,X'^*/S^*)\cong \mathcal{H}^p(X^*,X'^*/S^*),
\end{equation}
which is a simple variant of the relative de Rham theorem, for holomorphic forms vanishing on the boundary. Let $\gamma(t)\in \cup_{t\in S^*} H_p(X_t,X'_t;\mathbb{C})$ be a locally constant (horizontal) section of the relative homology bundle, i.e. a section of the local system $(R^pf_*\mathbb{C}_{X^*\setminus X'^*})^*$, dual to the local system $R^pf_*\mathbb{C}_{X^*\setminus X'^*}=\ker D_{f,H}$.  Let $\omega \in \mathcal{H}^p_{dR}(X^*,X'^*/S^*)$ be a relative cohomology class represented by a holomorphic form $\omega \in \Omega^{p}_{X^*/S^*}(X')$. Then, the integral:
\[I(t)=\int_{\gamma(t)}\omega\]
is well defined (because $\omega$ vanishes on the boundary $X'$), it is nondegenerate (it takes zero values on relatively exact forms and relative boundaries) and it is also a holomorphic (multivalued) function of $t \in S^*$. The verification of the holomorphicity comes from a relative version of the Leray residue formula:
\begin{equation}
\label{lrf}
\int_{\gamma(t)}\omega=\frac{1}{2\pi i}\int_{\sigma \gamma(t)}\frac{df\wedge \omega}{f-t},
\end{equation}
where the relative Leray boundary operator 
\[\sigma: H_p(X_t,X'_t;\mathbb{C})\rightarrow H_{p+1}(X\setminus X_t, X'\setminus X'_t; \mathbb{C})\] 
is defined as follows: choose a tubular neighborhood $N$ of the fiber $X_t$ whose intersection with the boundary $X'$ gives a tubular neighborhood $N'$ of the subfiber $X'_t$ (such a choice is always possible by the transversality of $X_t$ with $X'$). The image of a relative cycle $\gamma(t)$ under $\sigma$ is then the relative cycle obtained by the preimage of $\gamma(t)$ under the natural projection (fibration by circles $S^1$) of the boundary of the tubular neighborhood $\partial N$ over $X_t$. In particular, the relative Leray boundary operator is such that it makes the following diagram of long exact homology sequences commutative:
\[\begin{CD}
\label{cd6}
 \vdots   @. \vdots  \\
@VVV           @VVV               \\
H_p(X_t;\mathbb{C}) @>>>H_{p+1}(X\setminus X_t;\mathbb{C}) \\
 @VVV           @VVV                  \\
H_p(X_t,X'_t;\mathbb{C}) @> \sigma >> H_{p+1}(X\setminus X_t, X'\setminus X'_t;\mathbb{C})\\
 @VVV           @VVV               \\
H_{p-1}(X'_t;\mathbb{C}) @>>> H_{p}(X'\setminus X'_t;\mathbb{C})\\
 @VVV           @VVV               \\  
\vdots   @. \vdots  \\
\end{CD}\]
where the upper and lower arrows are the ordinary Leray boundary operators. The proof of the formula (\ref{lrf}) is then the same as in the ordinary case. From this it follows that indeed the function $I(t)$ is holomorphic in $t$, from which we immediately obtain the isomorphism (\ref{dr0}):
\[\mathcal{H}^p_{dR}(X^*,X'^*/S)\cong (\mathcal{H}_p(X^*,X'^*/S^*))^*\cong \mathcal{H}^p(X^*,X'^*/S^*).\]
The analytic Gauss-Manin connection on the relative de Rham cohomology sheaves $\mathcal{H}^p_{dR}(X^*,X'^*/S^*)$ can now be defined as follows: calculate first the formula of the derivative of $I(t)$:
\[I'(t)=\frac{d}{dt}\int_{\gamma(t)}\omega=\frac{1}{2\pi i}\int_{\sigma \gamma(t)}\frac{df\wedge \omega}{(f-t)^2}=\frac{1}{2\pi i}\int_{\sigma \gamma(t)}\frac{d\omega}{f-t}=\]
\[=\frac{1}{2\pi i}\int_{\sigma \gamma(t)}\frac{df\wedge \eta}{f-t}=\int_{\gamma(t)}\eta,\]
where $\eta \in \Omega^{p}_{X^*/S^*}(X')$ is the Gelfand-Leray form of $d\omega$:
\[\eta=\frac{d\omega}{df},\]
defined by the condition $d\omega=df\wedge \eta$ (because $\omega$ is relatively closed). Notice now that the condition $0=d(d\omega)=df\wedge d\eta$ implies the existence of a $p$-form vanishing on the boundary $\alpha \in \Omega^{p}_{X}(X')$, such that $d\eta=df\wedge \alpha$ (this can be verified for example by taking local coordinates). Thus, we may define a map:
\[D_{f,H}:\mathcal{H}^p_{dR}(X^*,X'^*/S^*)\rightarrow \mathcal{H}^p_{dR}(X^*,X'^*/S^*),\]
by the rule:
\[D_{f,H}\omega=\frac{d\omega}{df}=\eta,\]
which, as is easily verified, it is $\mathbb{C}$-linear and satisfies the Leibniz rule over $\mathcal{O}_{S^*}$, i.e. it defines a connection on $\mathcal{H}^p_{dR}(X^*,X'^*/S^*)$. Moreover, by the formula of the derivative $I'(t)$ above, the connection $D_{f,H}$ coincides with the topological Gauss-Manin connection on $\mathcal{H}^p(X^*,X'^*/S^*)$. We will call it the relative (analytic) Gauss-Manin connection.

Now we will show that for all $p<n$, the relative Gauss-Manin connection $D_{f,H}$ can be extended at the origin $0 \in S$, i.e. to a map:
\[D_{f,H}:H^p(\Omega^{\bullet}_{f}(H))\rightarrow H^p(\Omega^{\bullet}_{f}(H))\]
defined by the same rule:
\[D_{f,H}\omega=\frac{d\omega}{df}=\eta.\]
To do this, it suffices to verify that the germ of the $p$-form $\eta \in \Omega^p_{f}(H)$ is indeed relatively closed. This follows from the lemma below, which is a relative analog of the de Rham division lemma \cite{Der}: 
\begin{lem}
\label{rdrd}
For all $p\leq n$ and any relative form $\omega \in \Omega^{p}(H)$ such that $df\wedge \omega=0$, there exists a $(p-1)$-form $\alpha \in \Omega^{p-1}(H)$ such that $\omega=df\wedge \alpha$.
\end{lem}
\begin{proof}
It follows from the fact that the de Rham division lemma holds for both $f$ and $f|_H$ because their singularities are isolated. Briefly, consider the Koszul complexes $K^{\bullet}_{f}=(\Omega^{\bullet}, df\wedge)$, $K^{\bullet}_{f}(H)=(\Omega^{\bullet}(H), df\wedge)$ and $i_*K^{\bullet}_{f|_H}=(i_*\Omega^{\bullet}_H, df\wedge)$ and the corresponding short exact sequence:
\[0\rightarrow K^{\bullet}_{f}(H)\rightarrow K^{\bullet}_{f}\rightarrow i_*K^{\bullet}_{f|_H}\rightarrow 0.\]
The statement of the lemma is then equivalent to the cohomologies $H^p(K^{\bullet}_{f}(H))$ being all zero for $p\leq n$. This follows in turn by the long exact cohomology sequence and the fact that $H^p(K^{\bullet}_f)$ and $H^{p-1}(i_*K^{\bullet}_{f|_H})$ are both zero for all $p\leq n$. Indeed, the first statement is equivalent to the ordinary de Rham division lemma for $f$, while the second statement follows from the natural isomorphism\footnote{which is the isomorphism (\ref{gro}) with the direct image functor $f_*$ replaced with the global sections functor $\Gamma$ and the complex $\mathcal{F}^{\bullet}$ with the Koszul complex $K^{\bullet}_{f|_H}$.}:
\[H^{p-1}(i_*K^{\bullet}_{f|_H})\cong H^{p-1}(K^{\bullet}_{f|_H})\]
and the de Rham division lemma for the restriction $f|_H$
\end{proof}
\begin{rem}
\label{r1}
It follows from the argument above that the nonzero cohomologies of the Koszul complexes are in degree $n+1$:
\[H^{n+1}(K^{\bullet}_f)=\Omega^{n+1}_{f}, \hspace{0.3cm} H^n(K^{\bullet}_{f|_H})=\Omega^{n}_{f|_H},\]
\[H^{n+1}(K^{\bullet}_{f}(H))=\Omega^{n+1}_{f}(H)\]
and thus, there is a short exact sequence:
\begin{equation}
\label{kses}
0\rightarrow \Omega^{n}_{f|_H}\stackrel{df\wedge}{\rightarrow}\Omega^{n+1}_{f}(H)\rightarrow \Omega^{n+1}_f\rightarrow 0.
\end{equation}
But after a choice of coordinates $(x,y_1,...y_n)$ for which $H=\{x=0\}$ and division with the form $\omega=dx\wedge dy_1\wedge ... \wedge dy_n$, the short exact sequence above reduces to a short exact sequence of the corresponding local algebras (c.f. \cite{Szp}):
\[0\rightarrow \mathcal{Q}_{f|_H}\rightarrow \mathcal{Q}_{f,H}\rightarrow \mathcal{Q}_f\rightarrow 0.\]
This gives also another proof of the formula for the Milnor number of a boundary singularity:
\[\mu_{f,H}=\mu_{f}+\mu_{f|_H}.\] 
\end{rem}
Thus, the map $D_{f,H}$ can be indeed extended at the origin and consequently it defines a connection in the usual sense for all $p<n$ as expected. Attempting now to extend the relative Gauss-Manin connection at the origin for $p=n$ we come to the obstruction that the form $d\eta=d(\frac{d\omega}{df})$ may not be relatively closed, being of maximal degree $n+1$. To study the Gauss-Manin connection in this case we may, following Brieskorn \cite{B}, define two extensions of the cohomology module $H^n(\Omega^{\bullet}_{f}(H))$ (the relative Brieskorn modules) as follows: denote by $H_{f,H}:=H^n(\Omega^{\bullet}_{f}(H))$ and consider the natural inclusion of this module in the cokernel of the differential $d:\Omega^{n-1}_{f}(H)\rightarrow \Omega^{n}_{f}(H)$:
\[H_{f,H}\subset H'_{f,H}:=\frac{\Omega^{n}_{f}(H)}{d\Omega^{n-1}_{f}(H)}\cong \frac{\Omega^n(H)}{df\wedge \Omega^{n-1}(H)+d\Omega^{n-1}(H)}.\]
Consider now multiplication by $df\wedge $ on $H'_{f,H}$. It defines an isomorphism:
\[H'_{f,H}\stackrel{df\wedge}{\xrightarrow{\sim}}\frac{df\wedge \Omega^{n}(H)}{df\wedge d\Omega^{n-1}(H)}\]
and we thus obtain another natural inclusion:
\[H'_{f,H}\stackrel{df\wedge}{\subset} H''_{f,H}:=\frac{\Omega^{n+1}}{df\wedge d\Omega^{n-1}(H)}.\]
We have thus a sequence of inclusions of $\mathbb{C}\{f\}$-modules:
\[H_{f,H}\subset H'_{f,H}\subset H''_{f,H},\]
whose cokernels are both isomorphic to the same $\mu_{f,H}$-dimensional $\mathbb{C}$-vector space:
\[\frac{H'_{f,H}}{H_{f,H}}\stackrel{d}{\xrightarrow{\sim}}\Omega^{n+1}_{f}(H), \hspace{0.3cm} \frac{H''_{f,H}}{H'_{f,H}}\cong \Omega^{n+1}_{f}(H).\]
Hence, we may view these modules as defining lattices in the same $\mu_{f,H}$-dimensional vector space over the field of quotients $\mathbb{C}(f)$ of $\mathbb{C}\{f\}$:
\[\mathcal{M}_{f,H}=H_{f,H}\otimes_{\mathbb{C}\{f\}}\mathbb{C}(f)= H'_{f,H}\otimes_{\mathbb{C}\{f\}}\mathbb{C}(f)= H''_{f,H}\otimes_{\mathbb{C}\{f\}}\mathbb{C}(f)\]
In analogy with the ordinary case we call the modules $H'_{f,H}$ and $H''_{f,H}$ the relative Brieskorn modules (or lattices) of the boundary singularity $(f,H)$. 

Now, using the relative Brieskorn modules we may extend the map $D_{f,H}$ to two maps (which we denote by the same symbol):
\[D_{f,H}:H_{f,H}\rightarrow H'_{f,H}, \hspace{0.3cm} D_{f,H}\alpha=\frac{d\alpha}{df}=\eta,\]
\[D_{f,H}:H'_{f,H}\rightarrow H''_{f,H}, \hspace{0.3cm} D_{f,H}\eta=D_{f,H}(df\wedge \eta)=d\eta,\]
which, as is easily verified, are $\mathbb{C}$-linear and satisfy the Leibniz rule over $\mathbb{C}\{f\}$ (they define ``connections'' on the corresponding pairs of modules in the sense of Malgrange \cite{Mal}). For these maps we have first the following important proposition:
\begin{prop}
\label{p}
The maps $D_{f,H}$ defined above induce isomorphisms of the underlying $\mathbb{C}$-vector spaces, i.e. there exists a commutative diagram:
\[\begin{CD}
\label{cd7}
H'_{f,H} @>D_{f,H}>\sim > H''_{f,H} @>>> \Omega^{n+1}_{f}(H)\\
 @A D_{f,H} A \wr A           @A D_{f,H} A \wr A               @|               \\
H_{f,H} @>D_{f,H}>\sim > H'_{f,H} @>>> \Omega^{n+1}_{f}(H)\\
\end{CD}\]
\end{prop}
\begin{proof}
We will show that the map $D_{f,H}:H'_{f,H}\rightarrow H''_{f,H}$ is indeed an isomorphism (for the other map see Proposition \ref{t4}). It is obviously surjective since for any $\omega \in \Omega^{n+1}$ representing a class in $H''_{f,H}$ there exists a form $\eta \in \Omega^{n}_H$ such that $\omega=d\eta$ (by the Poincar\'e lemma for $\Omega^{\bullet}(H)$). To show that it is injective, let $D_{f,H}\eta=d\eta=0$. This means that for a representative $d\eta \in \Omega^{n+1}$ of the class $d\eta$ there exists a form $h\in \Omega^{n-1}(H)$ such that $d\eta=df\wedge dh$. Thus $\eta=df\wedge h+dg$ for some $g\in \Omega^{n-1}(H)$, i.e. the class of $\eta$ is indeed zero in $H'_{f,H}$. 
\end{proof}
Despite the fact that these maps do not define connections in the ordinary sense, it follows that they induce the same meromorphic connection $D_{f,H}$ on the localisation $\mathcal{M}_{f,H}$ of the relative Brieskorn modules:
\[D_{f,H}:\mathcal{M}_{f,H}\rightarrow \mathcal{M}_{f,H}\]
defined as follows: let $\omega \in \Omega^{n+1}$ be a representative of a class in $H''_{f,H}$. Since the boundary singularity $(f,H)$ is isolated there exists a natural number $k<\infty$  such that $f^k\omega=df\wedge \eta$, where $\eta \in \Omega^{n}(H)$. Then $D_{f,H}(f^k\omega)=D_{f,H}(df\wedge \eta)=d\eta$ and by the Leibniz rule we obtain in $\mathcal{M}_{f,H}$:
\[D_{f,H}\omega=\frac{d\eta}{f^k}-k\frac{\omega}{f}.\]
It is easy now to verify that the map thus defined is $\mathbb{C}$-linear and satisfies the Leibniz rule over $\mathbb{C}(f)$, i.e. it indeed defines a connection on $\mathcal{M}_{f,H}$, with a pole of degree at most $k$ at the origin. 

\begin{rem}
\label{r2}
In the next section we will show that the relative Gauss-Manin connection thus defined is regular, i.e. there exists a (meromorphic) change of coordinates such that $D_{f,H}$ has a pole of degree at most 1 at the origin. The residue $\text{Res}_0D_{f,H}$ of the connection is then the constant matrix $\Gamma$ in the representation:
\[y'=(\frac{\Gamma}{t}+\tilde{\Gamma}(t))y,\]
of the differential system of horizontal sections in this basis, where $\tilde{\Gamma}(t)$ is a holomorphic matrix. Since the characteristic polynomial of the relative Picard-Lefschetz monodromy $T_{f,H}$ is integral, it is constant under variations of $t$ and thus its roots $\lambda_j$ coincide with the numbers $e^{-2\pi i\alpha_j}$, where $\alpha_j$ are the eigenvalues of $\text{Res}_0D_{f,H}$. Moreover, one may show\footnote{following for example the same construction as in \cite{B}} that the connection $D_{f,H}$ is algebraically defined, i.e. that for any automorphism $\phi:\mathbb{C}\rightarrow \mathbb{C}$ the following relation holds:
\[D_{\phi f,H}=\phi \circ D_{f,H}.\]
It follows then from the solution of Hilbert's VII problem that the eigenvalues $\alpha_j$ of $\text{Res}_0D_{f,H}$ are rational numbers and thus, the eigenvalues of the relative monodromy operator $T_{f,H}$ are indeed roots of unity.
\end{rem}

\subsubsection{Asymptotics of Integrals along Vanishing Cycles: the Relative Sebastiani Theorem and Regularity of the Relative Gauss-Manin Connection}

We give here a direct extension of some results obtained by Malgrange in \cite{Mal}, concerning the asymptotics of integrals of holomorphic forms along relative vanishing cycles. First we will need the following estimate which we will use to prove the relative Sebastiani theorem as well as the regularity theorem for the relative Gauss-Manin connection:
\begin{prop}
\label{m}
For any relative $n$-form $\omega \in \Omega^n_{X/S}(X')$ and any section $\gamma(t) \in H_n(X_t,X'_t;\mathbb{C})$ in a sector containing the zero ray:
\[\lim_{t\rightarrow 0, \arg t=0}\int_{\gamma(t)}\omega=0.\]
\end{prop}
\begin{proof}
The proof is the same as in \cite{Mal} with simple modifications: let $\omega \in \Omega^n_{X}(X')$ represent the class of $\omega$. Fix a real $t_0>0$ and let $Y=f^{-1}([0,t_0])\subset X$, $Y'=f^{-1}([0,t])\cap X'=f'^{-1}([0,t_0])\subset X'$. Let $\gamma(t_0)$ be a relative $n$-cycle on $X_{t_0}$ and let $\Gamma$ be a representative. By the fact that  the pair $(X_{t_0},X'_{t_0})$ is contractible, it follows that the pair $(Y,Y')$ is contractible as well. Since $Y$ is semianalytic and $Y'$ is a semianalytic subset, we may find semianalytic triangulation of $Y$ such that both $Y'$ and $X_{t_0}$ are subcomplexes of $Y$ and such that $X'_{t_0}=X_{t_0}\cap Y'$ is a subcomplex of both $Y'$ and $X_{t_0}$  (c.f. \cite{Loj}). Thus, there exists a relative $(n+1)$-chain $\Delta$ such that $\Gamma=\partial \Delta$ (here the boundary operator $\partial$ is the one induced on the relative chains). By an immediate extension of Stokes-Herrera theorem \cite{He} for the relative case, we have that the integrals 
\[I(t_0)=\int_{\gamma(t_0)}\omega=\int_{\Gamma}\omega=\int_{\Delta}d\omega\]
are well defined. Consider now a relative $(n+1)$-chain $\Delta_t=f^{-1}([0,t])\cap \Delta$, $t\in (0,t_0]$. Then $\Delta=\Delta_t+\Delta'$ where $\Delta'$ is a relative $(n+1)$-chain on $f^{-1}([t,t_0])$ and $\partial \Delta'=\Gamma-\Gamma_t$. It follows that $\Gamma_t$ is a relative cycle representing $\gamma(t)$ and
\[I(t_0)=\int_{\Delta}d\omega=\int_{\Delta_t}d\omega +\int_{\Delta'}d\omega=\int_{\Delta_t}d\omega+\int_{\Gamma}\omega -\int_{\Gamma_t}\omega =\int_{\Delta_t}d\omega +I(t_0)-I(t),\]
i.e. 
\[I(t)=\int_{\Gamma_t}\omega=\int_{\Delta_t}d\omega.\]
But 
\[\lim_{t\rightarrow 0}\int_{\Delta_t}d\omega =\int_{\Delta_0}d\omega\]
where $\Delta_0=X_0\cap \Delta$ is a relative $n$-chain on $X_0$. By the fact that the restriction of $d\omega$ on the smooth part of $X_0$ is zero, it follows that $\lim_{t\rightarrow 0}I(t)=0$ as was asserted.
\end{proof}

As an immediate corollary of this proposition we obtain the following relative analog of the Sebastiani theorem \cite{S}:
\begin{thm}
\label{t3}
The relative Brieskorn module $H''_{f,H}$ (and thus $H'_{f,H}$ and $H_{f,H}$) is a free module of rank $\mu_{f,H}$.
\end{thm}
\begin{proof}
The proof is again the same as in \cite{Mal}. Briefly, let $H'^T_{f,H}$ and $H''^T_{f,H}$ be the torsion submodules of the corresponding Brieskorn modules with $H''^T_{f,H}\neq 0$. We have $D_{f,H}H'^T_{f,H}\subset H''^T_{f,H}$ and necessarily $H'^T_{f,H}\neq H''^T_{f,H}$ because then the restriction of $D_{f,H}$ will give a connection on $H'^T_{f,H}=H''^T_{f,H}$ and thus $H''^T_{f,H}=0$. Since $D_{f,H}:H'_{f,H}\rightarrow H''_{f,H}$ is an isomorphism (Proposition \ref{p}) it follows that there exists nonzero $\omega \in H'_{f,H}$ such that $\omega \notin H'^T_{f,H}$ and $D_{f,H}\omega \in H''^T_{f,H}$.   After tensoring with $\mathbb{C}(f)$ we find a form $\omega \in \Omega^n_{f,H}$ such that its class $\omega \in H'_{f,H}\otimes_{\mathbb{C}\{f\}}\mathbb{C}(f)$ satisfies $\omega \neq 0$ and $D_{f,H}\omega=0$. But then, for any section $\gamma(t)\in H_n(X_t,X'_t;\mathbb{C})$ we have:
\[I'(t)=\frac{d}{dt}\int_{\gamma(t)}\omega=\int_{\gamma(t)}D_{f,H}\omega=0,\]
i.e. $I(t)$ is constant. From Proposition \ref{m} we have that $I(t)=0$ and thus $\omega=0$ in $H'_{f,H}\otimes_{\mathbb{C}\{f\}}\mathbb{C}(f)$ which is a contradiction. Thus $H''^T_{f,H}=0$ which proves the theorem.
\end{proof}

Now we will prove the following relative analog of the regularity theorem:
\begin{thm}
\label{t2}
The relative Gauss-Manin connection $D_{f,H}:\mathcal{M}_{f,H}\rightarrow \mathcal{M}_{f,H}$ is regular.
\end{thm}
\begin{proof}
The proof is again the same as in \cite{Mal}. Recall (c.f \cite{Del}) that the condition of regularity of a connection is equivalent to the fact that each of the components $I_j(t)$ of the (multivalued) solutions $I(t)=(I_1(t),...,I_{\mu_{f,H}}(t))^T$ of the differential system:
\begin{equation}
\label{pf}
\frac{dI}{dt}=\Gamma(t)I(t),
\end{equation}
where $\Gamma(t)$ is the connection matrix, is of moderate growth, i.e. for $t\rightarrow 0$ and in a fixed sector $a\leq \arg t\leq b$, $(a,b)\in \mathbb{R}^2$, there exist natural numbers $K$ and $N$ such that:
\[|I_j(t)|\leq K|t|^{-N}.\] 
Fixing a basis $\{\omega_1,...,\omega_{\mu_{f,H}}\} \in H'_{f,H}\otimes_{\mathbb{C}\{f\}}\mathbb{C}(f)$ we may consider for a locally constant section $\gamma(t)\in H_n(X_t,X'_t;\mathbb{C})$ the multivalued functions 
\[I_j(t)=\int_{\gamma(t)}\omega_j\]
and the corresponding vector-valued map $I(t)=(I_1(t),...,I_{\mu_{f,H}})$ as a solution of the equation (\ref{pf}) above (the Picard-Fuchs equation, expressing the condition of horizontality of the section $\gamma(t)$ with respect to the dual Gauss-Manin connection in a basis dual to $\omega_j$). Indeed, 
\[I'_j(t)=\int_{\gamma(t)}D_{f,H}\omega_j=\int_{\gamma(t)}\sum_{i=1}^{\mu_{f,H}}\Gamma_{ij}(f)\omega_i=\sum_{i=1}^{\mu_{f,H}}\Gamma_{ij}(t)I_i(t).\]
Thus, to prove regularity it suffices to prove that these integrals are indeed of moderate growth. This follows immediately from Proposition \ref{m} applied to $I_j(t)$ and an application of the Phragm\'en-Lindel\"of theorem for the strip $a\leq \arg t \leq b$ as in \cite{Mal}.
\end{proof}

Combining the regularity of the relative Gauss-Manin connection with the relative monodromy theorem we may obtain a more exact calculation of the asymptotics of integrals of holomorphic forms along the relative vanishing cycles of the boundary singularity. Let us define first some natural trivilisations of the cohomology bundle $R^nf_*\mathbb{C}_{X^*\setminus X'^*}=\cup_{t\in S^*} H^n(X_t,X'_t;\mathbb{C})$. Notice that from Theorem \ref{t1} a basis $\{\alpha_1,...,\alpha_{\mu_{f,H}}\}$ of the cohomology module $H^n(\Omega^{\bullet}_{f}(H))$ extends to a basis of the locally free sheaf $\mathcal{H}^n_{dR}(X,X'/S)$ in a neighborhood of the origin and each fiber $\mathcal{H}^n_{dR}(X,X'/S)_t$ is isomorphic to the cohomology $H^n(X_t,X'_t;\mathbb{C})\otimes_{\mathbb{C}_{S^*}}\mathcal{O}_{S^*,t}$ for $t\neq 0$. Thus, the map $t\in S^* \mapsto \{\alpha_1|_{X_t},...,\alpha_{\mu_{f,H}}|_{X_t}\} \in H^n(X_t,X'_t;\mathbb{C})$ gives a trivilisation of the relative cohomology bundle. Consider now the sheafification of the first relative Brieskorn module $H'_{f,H}$:
\[\mathcal{H}'_{X,X'/S}:=\frac{f_*\Omega^n_{X/S}(X')}{d(f_*\Omega^{n-1}_{X/S}(X'))},\]
and the natural short exact sequence:
\[0\rightarrow \mathcal{H}^n_{dR}(X,X'/S)\rightarrow \mathcal{H}'_{X,X'/S}\stackrel{d}{\rightarrow} f_*\Omega^{n+1}_{X/S}(X')\rightarrow 0.\]
Since the sheaf on the right is concentrated at the origin $0\in S$, there is an isomorphism:
\[\mathcal{H}^n_{dR}(X^*,X'^*/S^*)\cong \mathcal{H}'_{X^*,X'^*/S^*},\]
and so, we may define a trivilisation of the cohomology bundle by starting from a basis of $H'_{f,H}$ instead, and in fact of $H'_{f,H}\otimes_{\mathbb{C}\{f\}}\mathbb{C}(f)$. Such a basis can be found in turn as follows (c.f. \cite{B} for the ordinary case): Let $\{\omega_1,...,\omega_{{\mu}_{f,H}}\}$ be a basis of the second relative Brieskorn module $H''_{f,H}$. Then division by $df$ gives a basis $\{\frac{\omega_1}{df},...,\frac{\omega_{\mu_{f,H}}}{df}\}$ of $H'_{f,H}\otimes_{\mathbb{C}\{f\}}\mathbb{C}(f)$.  If we consider now the sheafification of the second relative Brieskorn module $H''_{f,H}$:
\[\mathcal{H}''_{X,X'/S}:=\frac{f_*\Omega^{n+1}_X}{df\wedge d(f_*\Omega^{n-1}_{X/S}(X'))}\]
and the natural short exact sequence:
\[0\rightarrow \mathcal{H}'_{X,X'/S}\rightarrow \mathcal{H}''_{X,X'/S}\rightarrow f_*\Omega^{n+1}_{X/S}(X')\rightarrow 0,\] 
then, by the same argument as before, there is an isomorphism:
\[\mathcal{H}'_{X^*,X'^*/S^*}\cong \mathcal{H}''_{X^*,X'^*/S^*}.\]
By coherence and freeness of the Brieskorn module the basis $\{\omega_1,...,\omega_{\mu_{f,H}}\}$ extends to a basis of $\mathcal{H}''_{X^*,X'^*/S^*}$ in a neighborhood of the origin, so that $\{\frac{\omega_1}{df},...,\frac{\omega_{\mu_{f,H}}}{df}\}$ extends to a basis of $\mathcal{H}'_{X^*,X'^*/S^*}$ as well. It follows that the map $t\in S^*\mapsto \{\frac{\omega_1}{df}|_{X_t},...,\frac{\omega_{\mu_{f,H}}}{df}|_{X_t}\}\in H^n(X_t,X'_t;\mathbb{C})$ defines a trivilisation of the cohomology bundle. In fact, for any $\omega \in H''_{f,H}$, the holomorphic form $\frac{\omega}{df}|_{X_t}$ is nothing but the Poincar\'e residue at $X_t$ of the form $\frac{\omega}{f-t}$ :
\[\text{Res}_{X_t}(\frac{\omega}{f-t})=\frac{\omega}{df}|_{X_t}.\]
The map $t \in S^*\mapsto s[\omega](t)=\frac{\omega}{df}|_{X_t}\in H^n(X_t,X'_t;\mathbb{C})$ is what A. N. Varchenko called ``a geometric section'' (c.f. \cite{Var1} and also \cite{A00}, \cite{Ku} and references therein). Thus,in order to obtain a triviliasation of the relative cohomology bundle, it suffices to find a basis of $H''_{f,H}$ and by Nakayama's lemma, a basis of the $\mu_{f,H}$-dimensional $\mathbb{C}$-vector space $\frac{H''_{f,H}}{fH''_{f,H}}$ (c.f. Example \ref{ex} below for the quasihomogeneous case).

Fix now a form $\omega \in H''_{f,H}$ and denote by:
\[I_{\omega,\gamma}(t)=<s[\omega](t),\gamma(t)>=\int_{\gamma(t)}\frac{\omega}{df},\]
where $\gamma(t)\in \cup H_{n}(X_t,X'_t;\mathbb{C})$ is a locally constant section of the relative homology bundle. The theorem below is a relative analog of the classical theorem on the asymptotics of integrals obtained by Malgrange \cite{Mal} and others (see again \cite{A00}, \cite{Ku} and references therein):

\begin{thm}
\label{t4}
For $|t|$ sufficiently small there is a convergent expansion in each sector of $\arg t$:
\[I_{\omega,\gamma}(t)=\sum_{\alpha,k}a_{\alpha,k}t^{\alpha}\frac{(lnt)^k}{k!},\]
where:
\begin{itemize}
\item[(i.)] $a_{\alpha,k}$ are vectors in $\mathbb{C}^{\mu_{f,H}}$,
\item[(ii.)] the numbers $\alpha$ are rational numbers $>-1$ which belong in a set of arithmetic progressions with the property that $\lambda=e^{-2\pi i\alpha}$ is an eigenvalue of the relative Picard-Lefschetz monodromy operator in relative homology $H_{n}(X_t,X'_t;\mathbb{C})$,
\item[(iii.)] the numbers $k$ are integers $0\leq k \leq N$ where $N$ is the maximal size of Jordan blocks of the relative monodromy operator. In particular, if the size of the Jordan blocks corresponding to the eigenvalue $\lambda=e^{-2\pi i \alpha}$ is $\leq r$ then $0\leq k \leq r$.
\end{itemize}
\end{thm} 
\begin{proof}
Let $\eta \in \mathcal{H}'_{X,X'/S}$ be a local section of the Brieskorn module such that $D_{f,H}\eta=d\eta=\omega \in \mathcal{H}''_{X,X'/S}$. Then
\begin{equation}
\label{der}
I_{\omega,\gamma}(t)=\int_{\gamma}\frac{d\eta}{df}=\frac{d}{dt}\int_{\gamma(t)}\eta=V'_{\eta,\gamma}(t),
\end{equation}
where $V_{\eta,\gamma}(t)=\int_{\gamma(t)}\eta$. Since the map $D_{f,H}:\mathcal{H}'_{X,X'/S}\rightarrow \mathcal{H}''_{X,X'/S}$ is an isomorphism we may study first the expansion of the integral $V_{\eta,\gamma}(t)$ into asymptotic series. Let $\Lambda=\{\lambda_1,...,\lambda_{\mu_{f,H}}\}$ be the eigenvalues of the relative monodromy operator $T_{f,H}$ in cohomology $H^n(X_t,X'_t;\mathbb{C})$. Then $\{-\lambda_1,...,-\lambda_{\mu_{f,H}}\}$ are the eigenvalues of the relative monodromy operator $T^{f,H}$ in homology $H_n(X_t,X'_t;\mathbb{C})$. Let
\[\alpha_j=-\frac{1}{2\pi i}\ln \lambda_j\]
be the eigenvalues of the matrix $R$, where:
\[T^{f,H}=e^{2\pi i R}.\]
By the relative monodromy Theorem \ref{t}, the eigenvalues $\lambda_j=e^{-2\pi i \alpha_j}$ are roots of unity and so $\alpha_j$ are rational numbers defined modulo $\mathbb{Z}$. Denote by
\[L(\lambda_j)=\{\alpha_j^0, \alpha_j^0+1, \alpha_j^0+2,...\}\]
the arithmetic progression with one suitable value of $\alpha_j$. Let now $\{\omega_1,...,\omega_{\mu_{f,H}}\}$ be a local basis of the sheaf $\mathcal{H}'_{X,X'/S}$. Then the vector:
\[V(t)=(\int_{\gamma(t)}\omega_1,...,\int_{\gamma(t)}\omega_{\mu_{f,H}})^T\] 
is a solution of the Picard-Fuchs equation:
\[y'(t)=\Gamma^t(t)y(t),\]
where $\Gamma(t)$ is the connection matrix of the Gauss-Manin connection $D_{f,H}$ with respect to the basis $\{\omega_1,...,\omega_{\mu_{f,H}}\}$. A fundamental solution of this equation is given by the period matrix:
\[Y(t)=(\int_{\gamma_j(t)}\omega_i)_{i,j=1,...,\mu_{f,H}},\]
where $\{\gamma_1(t),...,\gamma_{\mu_{f,H}}(t)\}$ is a locally constant (horizontal) basis of the homology bundle $\cup H_n(X_t,X'_t;\mathbb{C})$. By well known theorems of differential equations (c.f. \cite{Del}), the period matrix can be represented in the form:
\[Y(t)=Z(t)t^R,\]
where $Z(t)$ is a single-valued holomorphic matrix on $S^*$. In particular, there is a constant matrix $C$ such that:
\[V(t)=Z(t)t^RC.\]
By the regularity Theorem \ref{t2}, the matrix $Z(t)$ is meromorphic at the origin. After a choice of a Jordan basis of the relative monodromy operator and the corresponding structure of the matrix $t^R$, we obtain an expansion:
\[V(t)=\sum_{\lambda \in \Lambda}\sum_{\alpha \in L(\lambda)}\sum_{k=0}^Na_{\alpha,k}t^{\alpha}\frac{(\ln t)^k}{k!}.\]
But by Proposition \ref{m} we have $\lim_{t\rightarrow 0}V(t)=0$ and thus all $\alpha \geq 0$. Moreover, if $\alpha =0$ then $a_{\alpha,k}=0$ for all $k\geq 1$.  Thus we have obtained the required expansion for the function $V(t)=V_{\eta,\gamma}(t)$. Then, by differentiating and using equation (\ref{der}) we obtain the required expansion for $I_{\omega,\gamma(t)}$. Thus, it suffices to prove only (ii.) But for $\alpha =0$ we have only constants in the expansion of $V(t)$ and thus all $\alpha > -1$ in the expansion of $I_{\omega,\gamma}(t)$. This finishes the proof.    
\end{proof}

\begin{exmp}[Quasihomogeneous Boundary Singularities.]
\label{ex}
By a quasihomogeneous boundary singularity $(f,H)$ we mean a quasihomogeneous germ $f$ at the origin of $\mathbb{C}^{n+1}$ such as its restriction $f|_H$ on the boundary $H=\{x=0\}$ is also quasihomogeneous. For example, all the simple boundary singularities in Arnol'd's list \cite{A1} are quasihomogeneous. It is easy to see that this is equivalent (analogously with \cite{Sa}) to that $f\in J_{f,H}$, where $J_{f,H}=(x\frac{\partial f}{\partial x}, \frac{\partial f}{\partial y_1},...,\frac{\partial f}{\partial y_n})$ is the Jacobian ideal of the boundary singularity. Equivalently this implies that $fH''_{f,H}=df\wedge H'_{f,H}$, i.e. 
\[fD_{f,H}H'_{f,H}=H'_{f,H},\]
that is, the operator $D_{f,H}=\frac{d}{df}$  has a pole of first order at the origin. The residue of the connection is then the linear operator between the $\mu_{f,H}$-dimensional $\mathbb{C}$-vector spaces:
\[\text{Res}_0D_{f,H}:\frac{H''_{f,H}}{fH''_{f,H}}\rightarrow \frac{H''_{f,H}}{fH''_{f,H}},\]
where:
\[\frac{H''_{f,H}}{fH''_{f,H}}\cong \frac{H''_{f,H}}{df\wedge H'_{f,H}}\cong \Omega^{n+1}_{f}(H)\cong \mathcal{Q}_{f,H}.\]
In partricular, by Nakayama's lemma, a monomial basis $e_m=x^{m_1}y_1^{m_2}...y_n^{m_{n+1}}$, $m=(m_1,...,m_{n+1})\in A$, $|A|=\mu_{f,H}$ of the vector space $\mathcal{Q}_{f,H}$, lifts to a basis $\omega_{m}=e_mdx\wedge dy^n$ of the relative Brieskorn module $H''_{f,H}$. An easy calculation shows that the forms $\omega_m$ are exactly the eigenvectors of the operator $fD_{f,H}$:
\[fD_{f,H}\omega_m=(\alpha(m)-1)\omega_m,\]
where:
\[\alpha(m)=\sum_{i=1}^{n+1}w_i(m_i+1),\]
and $(w_1,...w_{n+1})$ are the quasihomogeneous weights of $f$. Thus, the residue $\text{Res}_0D_{f,H}$ is a semisimple operator and in particular, the relative Picard-Lefschetz monodromy operator:
\[T_{f,H}=e^{-2\pi i \text{Res}_0D_{f,H}}\]
is semisimple, with eigenvalues:
\[\lambda_m=e^{-2\pi i \alpha(m)}.\]
Moreover, for any $(n+1)$-form $\omega$ and any locally constant relative cycle $\gamma(t) \in H_n(X_t,X'_t;\mathbb{C})$ there exists an asymptotic expansion for $t\rightarrow 0$:
\[I(t)=\int_{\gamma(t)}\frac{\omega}{df}=\sum_{\lambda \in \Lambda}\sum_{\alpha \in L(\lambda)}a_{\alpha}t^{\alpha-1},\]
where for each $\lambda_m$ $\alpha \in L(\lambda_m)=\{\alpha(m), \alpha(m)+1, \alpha(m)+2,...\}$ and $a_{\alpha}\in \mathbb{C}^{\mu_{f,H}}$. 

Let us calculate the numbers $\alpha(m)$ for the $A_{k}$, $B_{k}$, $C_{k}$ and $F_4$ singularities on the plane $\mathbb{C}^2$ with boundary $H=\{x=0\}$, i.e. the simple boundary singularities in Arnol'd's list \cite{A1}:
\begin{itemize}
\item[$A_{k}$:] The normal form is: $f=x+y^{k+1}$, $k=\mu_{f,H}\geq 1$. It is quasihomogeneous with weights $(w_1=1, w_2=\frac{1}{k+1})$. The monomials $1,y,...,y^{k-1}$ form a basis of $\mathcal{Q}_{f,H}$ and thus:
\[H''_{f,H}=df\wedge H''_{f|_H}=\text{span}_{\mathbb{C}\{f\}}\{dx\wedge dy, ydx\wedge dy,...,y^{k-1}dx\wedge dy\}.\]
In particular:
\[\alpha(m)=\{\frac{k+2}{k+1},...,\frac{2k+1}{k+1}\}.\]
\item[$B_{k}$:] The normal form is: $f=x^k+y^{2}$, $k=\mu_{f,H}\geq 2$. It is quasihomogeneous with weights $(w_1=\frac{1}{k}, w_2=\frac{1}{2})$. The monomials $1,x,...,x^{k-1}$ form a basis of $\mathcal{Q}_{f,H}$ and thus:
\[H''_{f,H}=\text{span}_{\mathbb{C}\{f\}}\{dx\wedge dy, xdx\wedge dy,...,x^{k-1}dx\wedge dy\}.\]
In particular:
\[\alpha(m)=\{\frac{k+2}{2k},...,\frac{3k}{2k}=\frac{3}{2}\}.\]
\item[$C_{k}$:] The normal form is: $f=xy+y^{k}$,  $k=\mu_{f,H}\geq 2$. It is quasihomogeneous with weights $(w_1=\frac{k-1}{k}, w_2=\frac{1}{k})$. The monomials $1,y,...,y^{k-1}$ form a basis of $\mathcal{Q}_{f,H}$ and thus:
\[H''_{f,H}=\text{span}_{\mathbb{C}\{f\}}\{dx\wedge dy, ydx\wedge dy,...,y^{k-1}dx\wedge dy\}.\]
In particular:
\[\alpha(m)=\{1=\frac{k}{k},\frac{k+1}{k}...,\frac{2k-1}{k}\}.\]
\item[$F_4$:] The normal form is: $f=x^2+y^{3}$, $\mu_{f,H}=4$. It is quasihomogeneous with weights $(w_1=\frac{1}{2}, w_2=\frac{1}{3})$. The monomials $1,x,y,xy$ form a basis of $\mathcal{Q}_{f,H}$ and thus:
\[H''_{f,H}=\text{span}_{\mathbb{C}\{f\}}\{dx\wedge dy, xdx\wedge dy,ydx\wedge dy, xydx\wedge dy\}.\]
In particular:
\[\alpha(m)=\{\frac{5}{6},\frac{4}{3}, \frac{7}{6}, \frac{5}{3}\}.\]
\end{itemize}  
\end{exmp}
\begin{rem}
As it is easy to see, in all the examples above, the following splitting (in the category of $\mathbb{C}\{f\}$-modules) for the relative Brieskorn module is valid:
\[H''_{f,H}\cong H''_f\oplus df\wedge H''_{f|_H}.\]
In the next section we will show that this is a general fact for all isolated boundary singularities.
\end{rem}

\subsection{Relations between the Relative and Ordinary Brieskorn Modules}

In the previous section we showed the regularity of the relative Gauss-Manin connection $D_{f,H}$ and the freeness of the Brieskorn module $H''_{f,H}$ independently of the regularity of the ordinary Gauss-Manin connections $D_f$ and $D_{f|_H}$, and the freeness of the ordinary Brieskorn modules $H''_f$ and $H''_{f|_H}$ respectively. On the other hand we know from the short exact cohomology sequence (\ref{ses2}) that the relative cohomology module $H_{f,H}:=H^n(\Omega^{\bullet}_{f}(H))$ is an extension of the two ordinary cohomology modules $H_f:=H^n(\Omega^{\bullet}_f)$ and $H_{f|_H}:=H^{n-1}(\Omega^{\bullet}_{f|_H})$, i.e. there is a short exact sequence of free $\mathbb{C}\{f\}$-modules of finite type:
\begin{equation}
\label{ses3}
0\rightarrow H_{f|_H}\stackrel{\delta}{\rightarrow}H_{f,H}\rightarrow H_f\rightarrow 0.
\end{equation}  
 
Here we will show that the relative Brieskorn modules $H'_{f,H}$ and $H''_{f,H}$ are also extensions of the two ordinary Brieskorn modules:
\[H'_{f|_H}:=\frac{i_*\Omega^{n-1}_{f|_H}}{di_*\Omega^{n-2}_{f|_H}}\cong \frac{i_*\Omega^{n-1}_H}{df\wedge i_*\Omega^{n-2}_H+di_*\Omega^{n-2}_H}\stackrel{df\wedge}{\subset} H''_{f|_H}:=\frac{i_*\Omega^n_H}{df\wedge di_*\Omega^{n-2}_H},\]
\[H'_f:=\frac{\Omega^n_f}{d\Omega^{n-1}_f}\cong \frac{\Omega^n}{df\wedge \Omega^{n-1}+d\Omega^{n-1}}\stackrel{df\wedge}{\subset}H''_f:=\frac{\Omega^{n+1}}{df\wedge \Omega^{n-1}}.\]
The statement for $H'_{f,H}$ is proved in the proposition below and for $H''_{f,H}$ immediately after that:   
\begin{prop}
\label{pro}
There exist $\mathbb{C}\{f\}$-linear map $\delta'$ that makes the following diagram commutative:
\begin{equation}
\label{cd8}
\begin{CD}
0 @>>>H_{f|_H} @>\delta >> H_{f,H} @> p >> H_f @>>> 0 \\
@. @V D_{f|_H} V \wr V           @V D_{f,H} V \wr V               @V D_f V \wr V               \\
0 @>>>H'_{f|_H} @> \delta' >> H'_{f,H} @> p' >> H'_f @>>> 0\\
\end{CD}
\end{equation}
Moreover, there exists a $\mathbb{C}$-linear map $\delta''$ which extends the above diagram to a commutative diagram:
\begin{equation}
\label{cd9}
\begin{CD}
0 @>>>H'_{f|_H} @>\delta' >> H'_{f,H} @> p' >> H'_f @>>> 0 \\
@. @V D_{f|_H} V \wr V           @V D_{f,H} V \wr V               @V D_f V \wr V               \\
0 @>>>H''_{f|_H} @> \delta'' >> H''_{f,H} @> p'' >> H''_f @>>> 0\\
\end{CD}
\end{equation}
\end{prop} 
\begin{proof}
Let us prove first the claim for the diagram (\ref{cd8}). It depends on the algebraic definition of the Gauss-Manin connections involved, i.e. as connecting homomorphisms in certain long exact cohomology sequences (c.f. \cite{Mal} for the ordinary case). More specifically, consider the stalk at the origin of the diagram of short exact sequences (\ref{cd3}):
\begin{equation}
\label{cd10}
\begin{CD}
@. 0   @. 0 @. 0 \\
@. @VVV           @VVV               @VVV               \\
0 @>>>df\wedge \Omega^{\bullet-1}(H) @>>> df\wedge \Omega^{\bullet-1} @>>> df\wedge i_*\Omega^{\bullet-1}_H @>>> 0 \\
@. @VVV           @VVV               @VVV               \\
0 @>>>\Omega^{\bullet}(H) @>>> \Omega^{\bullet} @>>> i_*\Omega^{\bullet}_H @>>> 0\\
@. @VVV           @VVV               @VVV               \\
0 @>>> \Omega^{\bullet}_{f}(H) @>>>  \Omega^{\bullet}_{f} @>>>  i_*\Omega^{\bullet}_{f|_H} @>>> 0\\
@. @VVV           @VVV               @VVV               \\
@. 0   @. 0 @. 0 \\
\end{CD}
\end{equation}
Taking the corresponding long exact cohomology sequences we obtain a commutative diagram whose part containing the corresponding connecting homomorphisms is depicted below:
\begin{equation}
\label{cd11}
\begin{CD}   
   @.       @VVV              @VVV    @VVV           @. \\
 @>>> H^{p-1}(\Omega^{\bullet}_{f|_H}) @> \delta >>  H^p(\Omega^{\bullet}_{f}(H)) @>>>  H^p(\Omega^{\bullet}_{f}) @>>> \\
 @. @V \partial_{f|_H}VV           @V \partial_{f,H} VV               @V \partial_f VV               \\
 @>>>H^p(df\wedge i_*\Omega^{\bullet}_H) @>\delta' >> H^{p+1}(df\wedge \Omega^{\bullet}(H)) @>>> H^{p+1}(df\wedge \Omega^{\bullet}) @>>>  \\
@. @VVV           @VVV               @VVV              \\
 @>>>H^p(\Omega^{\bullet}_H) @>\partial >> H^{p+1}(\Omega^{\bullet}(H)) @>>> H^{p+1}(\Omega^{\bullet}) @>>> \\
@. @VVV           @VVV               @VVV               \\
 @>>> H^p(\Omega^{\bullet}_{f|_H}) @>\delta >>  H^{p+1}(\Omega^{\bullet}_{f}(H)) @>>>  H^{p+1}(\Omega^{\bullet}_{f}) @>>>  \\
@. @VVV           @VVV               @VVV               \\
\end{CD}
\end{equation}
Consider now multiplication by $df\wedge $ in each of the complexes $\Omega^{\bullet}$, $\Omega^{\bullet}(H)$ and $i_*\Omega^{\bullet}_H$. By the relative de Rham division lemma \ref{rdrd} it induces, for all $p\leq n$ a commutative diagram:
\begin{equation}
\label{cd12}
\begin{CD}
0 @>>>\Omega^{p}_{f}(H) @>>> \Omega^{p}_f @>>> \Omega^{p}_{f|_H} @>>> 0\\
@. @V df\wedge V \wr V           @V df\wedge V\wr V               @Vdf\wedge V\wr V               \\
0 @>>>df\wedge \Omega^{p}(H) @>>> df\wedge \Omega^{p} @>>> df\wedge i_*\Omega^{p}_{H} @>>> 0 \\
\end{CD}
\end{equation}
where the vertical arrows are isomorphisms. Since $df\wedge $ commutes with each of the differentials in the relative complexes we obtain isomorphisms in cohomologies for all $p$:
\[H^p('\Omega^{\bullet}_f)\cong H^{p+1}(df\wedge \Omega^{\bullet}), \hspace{0.3cm} H^p('\Omega^{\bullet}_{f}(H))\cong H^{p+1}(df\wedge \Omega^{\bullet}(H)),\]
\[H^{p-1}('\Omega^{\bullet}_{f|_H})\cong H^{p}(df\wedge i_*\Omega^{\bullet}_H)\cong H^p(df'\wedge \Omega^{\bullet}_H),\]
where $'\Omega^{\bullet}_f$, $'\Omega^{\bullet}_{f}(H)$ and $'\Omega^{\bullet}_{f|_H}$ are the complexes $\Omega^{\bullet}_f$, $\Omega^{\bullet}_{f}(H)$ and $\Omega^{\bullet}_{f|_H}$ with their last terms replaced by zero. 
Putting these back in the diagram (\ref{cd11}) we obtain:
\begin{equation}
\label{cd13}
\begin{CD}   
   @.       @VVV              @VVV    @VVV           @. \\
 @>>> H^{p-1}(\Omega^{\bullet}_{f|_H}) @> \delta >>  H^p(\Omega^{\bullet}_{f}(H)) @>>>  H^p(\Omega^{\bullet}_{f}) @>>> \\
 @. @V D_{f|_H} V \wr V           @V D_{f,H} V \wr V               @V D_f V\wr V               \\
 @>>>H^{p-1}('\Omega^{\bullet}_{f|_H}) @>\delta' >> H^p('\Omega^{\bullet}_{f}(H)) @>>> H^p('\Omega^{\bullet}_{f}) @>>>  \\
@. @VVV           @VVV               @VVV              \\
 @>>>H^p(\Omega^{\bullet}_H) @>\partial >> H^{p+1}(\Omega^{\bullet}(H)) @>>> H^{p+1}(\Omega^{\bullet}) @>>> \\
@. @VVV           @VVV               @VVV               \\
 @>>> H^p(\Omega^{\bullet}_{f|_H}) @>\delta >>  H^{p+1}(\Omega^{\bullet}_{f}(H)) @>>>  H^{p+1}(\Omega^{\bullet}_{f}) @>>>  \\
@. @VVV           @VVV               @VVV               \\
\end{CD}
\end{equation}
where the map $\delta'$ is the connecting homomorphism in the long exact cohomology sequence induced by the short exact sequence:
\[0\rightarrow '\Omega^{\bullet}_{f}(H)\rightarrow '\Omega^{\bullet}_{f}\rightarrow '\Omega^{\bullet}_{f|_H}\rightarrow 0,\]
and it is thus $\mathbb{C}\{f\}$-linear. An easy calculation shows also that it is defined by the same rule with $\delta$. The first series of vertical maps in (\ref{cd13}) are the corresponding Gauss-Manin connections which are obtained as the composition of the maps in (\ref{cd11}) $\partial_{f|_H}$, $\partial_{f,H}$ and $\partial_f$ respectively, with the following isomorphisms: 
\[H^p('\Omega^{\bullet}_f)\cong H^p(\Omega^{\bullet}_f), \hspace{0.3cm} H^p('\Omega^{\bullet}_{f}(H))\cong H^p(\Omega^{\bullet}_f(H)),\]
\[H^{p-1}('\Omega^{\bullet}_{f|_H})\cong H^{p-1}(\Omega^{\bullet}_{f|_H}),\]
for all $p<n$, whereas for $p=n$:
\[H^n('\Omega^{\bullet}_f)\cong H'_f, \hspace{0.3cm} H^n('\Omega^{\bullet}_{f}(H))\cong H'_{f,H},\]
\[H^{n-1}('\Omega^{\bullet}_{f|_H})\cong H'_{f|_H}.\]
But for all $p<n$ all the cohomologies (except the zero ones) in the diagram (\ref{cd13}) above are zero, while for $p=n$ we obtain the commutative diagram (\ref{cd8}).  Finally, to obtain the commutative diagram (\ref{cd9}) it suffices to set 
\[\delta''=D_{f,H}\delta' D_{f|_H}^{-1}, \hspace{0.3cm} p''=D_fp'D_{f,H}^{-1}.\]
The map $\delta''$ takes a class $\omega \in H''_{f|_H}$ to the class of the differential $d\bar{\omega} \in H''_{f,H}$, where $\bar{\omega} \in \Omega^{n}$ is a lift of a representative of $\omega$. It is obvious that this map is $\mathbb{C}$-linear. This finishes the proof.
\end{proof}

In the proposition above the map $\delta''$ is not $\mathbb{C}\{f\}$-linear and so the short exact sequence in the bottom row of diagram (\ref{cd9}) is only short exact for the underlying $\mathbb{C}$-vector spaces. To show that the relative Brieskorn module $H''_{f,H}$ is an extension of the two ordinary Brieskorn modules $H''_{f|_H}$, $H''_f$, we identify first $H''_{f|_H}$ with $D_{f|_H}H'_{f|_H}=dH'_{f|_H}$, which is a free $\mathbb{C}\{f\}$-module of rank $\mu_{f|_H}$. The inclusion $df\wedge d\Omega^{n-1}(H)\subset df\wedge d\Omega^{n-1}$ induces a natural projection $\pi: H''_{f,H}\rightarrow H''_f$ whose kernel is exactly the module $df\wedge dH'_{f|_H}$. By the fact that $H''_f$ is free, we obtain a split short exact sequence of $\mathbb{C}\{f\}$-modules:
\[0\rightarrow dH'_{f|_H}\stackrel{df\wedge}{\rightarrow} H''_{f,H}\stackrel{\pi}{\rightarrow}H''_f\rightarrow 0,\] 
which is what we wanted to prove. This gives also another direct proof of the relative Sebastiani Theorem \ref{t3}:
\[H''_{f,H}\cong \mathbb{C}\{f\}^{\mu_{f,H}}.\]

As another immediate corollary of the above proposition we obtain a second proof of the  regularity Theorem \ref{t2} for the relative Gauss-Manin connection: indeed, both of the commutative diagrams (\ref{cd8}), (\ref{cd9}) give, after localisation, the following commutative diagram of finite dimensional $\mathbb{C}(f)$-vector spaces:
\begin{equation}
\begin{CD}
0 @>>>\mathcal{M}_{f|_H} @>>> \mathcal{M}_{f,H} @>>> \mathcal{M}_{f} @>>> 0\\
@. @V D_{f|_H}V V           @V D_{f,H} VV               @VD_f VV               \\
0 @>>>\mathcal{M}_{f|_H} @>>> \mathcal{M}_{f,H} @>>> \mathcal{M}_{f} @>>> 0 \\
\end{CD}
\end{equation}
The claim follows then from a well known proposition \cite{Del} according to which the connection $D_{f,H}$ is regular if and only if both $D_{f|_H}$ and $D_{f}$ are. 

\begin{exmp}
Let us describe as an example the $B_2$ singularity in Arnold's list in $\mathbb{C}^3$. This has the normal form:
\[f(x,y,z)=x^2+y^2+z^2, \hspace{0.2cm} H=\{x=0\}.\]
Here the relative monodromy captures both the trivial monodromy of the plane curve $f|_H=y^2+z^2$, as well as the ordinary Dehn twist. Indeed, the ordinary Brieskorn modules of $f|_H$ and $f$ are free of rank one:
\[H''_{f|_H}=\text{span}_{\mathbb{C}\{f\}}\{dy\wedge dz\}, \hspace{0.2cm} H''_{f}=\text{span}_{\mathbb{C}\{f\}}\{dx\wedge dy\wedge dz\},\]
and the relative Brieskorn module is also free of rank two:
\[H''_{f,H}\cong H''_f\oplus df\wedge H''_{f|_H}=\text{span}_{\mathbb{C}\{f\}}\{dx\wedge dy\wedge dz, xdx\wedge dy\wedge dz\}.\]
The spectrum is:
\[\alpha(m)=\{\frac{1}{2},1\}\]
and the relative monodromy matrix $T_{f,H}$ is semisimple with eigenvalues 
\[\lambda(m)=e^{-2\pi i \alpha(m)}=\{-1,1\}.\] 
Indeed, an easy calculation shows (by the quasihomogeneity of $(f,H)$) that the relative Gauss-Manin connection matrix is obtained by the system:
\[fD_{f,H}[\begin{array}{c}
		dx\wedge dy\wedge dz \\
		xdx\wedge dy\wedge dz
	     \end{array}]=[\begin{array}{cc}
			   \frac{1}{2} & 0 \\
			   0 & 1 \end{array}][\begin{array}{c}
				dx\wedge dy \wedge dz \\
				xdx\wedge dy\wedge dz 
				\end{array}],\]
and thus its monodromy is given by the matrix:
\[T_{f,H}= [ \begin{array}{cc}
-1 & 0 \\
0 & 1
\end{array}]. \]
\end{exmp}

\section{Boundary Singularities in Isochore Geometry}

We give here some more applications of the results obtained so far in isochore deformation theory, i.e. the deformation theory of boundary singularities with respect to a volume form. 

\subsection{Local Classification of Volume Forms and Functional Invariants}

We start first with a direct corollary of the finiteness and freeness of the relative Brieskorn module $H''_{f,H}$ concerning the classification of volume forms relative to diffeomorphisms tangent to the identity and preserving the boundary singularity $(f,H)$. Write $\mathcal{R}_{f,H}$ for the group of germs of these diffeomorphisms, i.e. such that:
\[\Phi^*f=f, \hspace{0.3cm} \Phi(H)=H,\]
\[\Phi(0)=0, \hspace{0.3cm} \Phi_*(0)=Id.\]
Two germs of volume forms at the origin will be called $\mathcal{R}_{f,H}$-equivalent (or equivalent for brevity) if they belong in the same orbit under the action of $\mathcal{R}_{f,H}$ in the space of germs of volume forms $\Omega^{n+1}_*$. The following theorem is a relative analog of a theorem obtained by J. -P. Fran\c{c}oise \cite{F0}, \cite{F1} (see also \cite{F2}) for the ordinary singularities, concerning the local normal forms of volume forms and their functional invariants:
\begin{thm}
\label{t5}
Two germs of volume forms are equivalent if and only if they define the same class in the relative Brieskorn module $H''_{f,H}$. In particular any germ of a volume form is equivalent to the form
\begin{equation}
\label{nf}
\omega=\sum_{i=1}^{\mu_{f,H}}c_i(f)\omega_i,
\end{equation}
where $c_i \in \mathbb{C}\{t\}$ and the classes of the forms $\omega_i$ form a basis of $H''_{f,H}$. 
\end{thm}
\begin{proof}
The one direction is trivial: if two germs of volume forms are equivalent then their Poincar\'e residues define the same cohomology class in each fiber $H^n(X_t,X'_t;\mathbb{C})$ of the cohomological Milnor fibration in a sufficiently small neighborhood of the origin. Indeed, since the diffeomorphism realising the equivalence is tangent to the identity, it induces the identity in the cohomology of each pair of fibers $(X_t,X'_t)$ with constant coefficients. It follows by the coherence and freeness of the Brieskorn module $H''_{f,H}$ that the diffeomorphism $\Phi$ induces the identity morphisms in both $H'_{f,H}$ and $H''_{f,H}$. The other direction is a trivial application of Moser's homotopy method, whose proof goes briefly as follows: consider a family of volume forms $\omega_s=\omega_0+sdf\wedge dg$, $s\in [0,1]$. Then the vector field $v_s$ defined by:
\[v_s\lrcorner \omega_s=g\wedge df\]
is a solution of the homological equation:
\[L_{v_s}\omega_s=-df\wedge dg\]
and thus, its time-1 map $\Phi_1$ is the desired diffeomorphism between $\omega_1$ and $\omega_0$. Choosing now a basis $\{\omega_1,...,\omega_{\mu_{f,H}}\}$ of $H''_{f,H}$ and $\omega_0$ as the representative of $\omega_1$ in this basis, then we obtain the normal form (\ref{nf}).
\end{proof}
\begin{rem}
Since the boundary singularity $(f,H)$ is isolated, we may always choose local coordinates $(x,y_1,...,y_n)$ such that in the theorem above $H=\{x=0\}$ and $f(x,y_1,...,y_n)$ is a polynomial of sufficiently high degree (by a relative analog of the determinacy theorem c.f. \cite{Mat}).  
\end{rem}

The case $\mu_{f,H}=\mu_{f|_H}=1$ i.e. the first occurring boundary singularity ($A_1$ in Arnol'd's list \cite{A1}), with normal form $f(x,y)=x+y_1^2+...+y_n^n$, $H=\{x=0\}$, is of special interest. The following theorem is a direct corollary of the above theorem and it may be interpreted as the relative analog of J. Vey's isochore Morse lemma \cite{V}. For its proof we follow \cite{F1} (for another proof see next section).

\begin{thm}
\label{t6}
Let $(f,H)$ be a boundary singularity such that the origin is a regular point for $f$ but nondegenerate critical point for the restriction $f|_H$ on the boundary. Then there exists a diffeomorphism $\Psi$, preserving both the boundary $H=\{x=0\}$ and the standard volume form $\omega=dx\wedge dy_1\wedge ...\wedge dy_n$, as well as a unique function $\psi \in \mathbb{C}\{t\}$, $\psi(0)=0$, $\psi'(0)=1$ such that
\begin{equation}
\label{e2}
\Psi^*f=\psi(x+y_1^2+...+y_n^2),
\end{equation} 
\end{thm} 
\begin{proof}
By Theorem \ref{t5} above we may choose a coordinate system $(x,y_1,...,y_n)$ such that $H=\{x=0\}$, $f(x,y)=x+y_1^2+...+y_n^2$ and $\omega=c(f)dx\wedge dy_1\wedge ...\wedge dy_n$, where $c\in \mathbb{C}\{t\}$ is a function, nonvanishing at the origin, $c(0)=1$. We will show that there exists a change of coordinates $\Psi(x,y_1,...,y_n)=(x',y_1',...,y_n')$ such that the pair $(f,H)$ goes to $(\psi(f), H)$ for some function $\psi$ and $\omega$ is reduced to normal form $dx\wedge dy_1\wedge ...\wedge dy_n$. To do this, we set $x'=xv(f)$, $y_i'=y_i\sqrt{v(f)}$, where $v\in \mathbb{C}\{t\}$ is some function with $v(0)=1$ (so $\Psi$ is indeed a boundary-preserving diffeomorphism tangent to the identity). With any such function $v$ we have $\Phi^*f=\psi(f)$, for some function $\psi(t)=tv(t)$ with $\psi(0)=0$ and $\psi'(0)=1$. Now it suffices to choose $v$ so that $\Phi_*$ has determinant equal to $c(f)$ , i.e. such that the following initial value problem is satisfied for the function $w=v^{\frac{n+2}{2}}$:
\begin{equation}
\label{a}
\frac{2}{n+2}tw'(t)+w(t)=c(t), \quad w(0)=1.
\end{equation}
As is easily verified this admits an analytic solution given by the formula:
\[w(t)=t^{-\frac{n+2}{2}}\int_0^t\frac{n+2}{2}s^{\frac{n}{2}}c(s)ds.\]
This also shows the uniqueness of the function $\psi(t)$, which can be written as:
\[\psi(t)=(\int_{0}^t\frac{n+2}{2}s^{\frac{n}{2}}c(s)ds)^{\frac{2}{n+2}}.\]
\end{proof}
 
\subsection{Isochore Versal Deformations of Boundary Singularities}

In \cite{G1}, M. D. Garay gave a different proof of Vey's isochore Morse lemma which, according to his results, is a simple consequence of an isochore version of Mather's versal unfolding theorem proved by him (as a positive answer to a question asked by Y. Colin de Verdi\`ere in \cite{C}). Here we will present the main parts of the proof of a relative version of the isochore unfolding theorem, i.e. for the isochore unfoldings of boundary singularities, by considering only the main modifications needed in order to adapt the same proof as in \cite{G1}.

To start recall that a deformation $F:(\mathbb{C}^{n+1} \times \mathbb{C}^k,0)\rightarrow (\mathbb{C},0)$ of a boundary singularity $(f,H)$ is just a deformation of $f$, $F(.;0)=f$, such that its restriction  $F|_H:(H\times \mathbb{C}^k,0)\rightarrow (\mathbb{C},0)$ on the boundary $H=\mathbb{C}^n \subset \mathbb{C}^{n+1}$, is a deformation of $f|_H$, $F|_H(.;0)=f|_H$.  To the deformation $F$ of the boundary singularity we associate its unfolding, i.e. the map:
\[\tilde{F}: (\mathbb{C}^{n+1}\times \mathbb{C}^k,0)\rightarrow (\mathbb{C}\times \mathbb{C}^k,0), \hspace{0.3cm} \tilde{F}(.;\lambda)=(F(.;\lambda),\lambda)\]
and accordingly we define also $\tilde{F}|_{H}$. Fix now the equation of the boundary $H=\{x=0\}$ and fix also a germ of a volume form $\omega=dx\wedge dy^n$ (where $dy^n=dy_1\wedge ...\wedge dy_n$) at the origin of $\mathbb{C}^{n+1}$. All the notions of Right-Left (or $\mathcal{A}$-)equivalence between deformations, versality, infinitesimal versality e.t.c. (c.f. \cite{A0}) carry over to the subgroup $\mathcal{A}_{\omega,H}$ of Right-Left equivalences, where the right diffeomorphism has to preserve both the boundary $H$ and the volume form $\omega$. In particular, a deformation $F$ (or the unfolding $\tilde{F}$) of a boundary singularity $(f,H)$ will be called isochore versal if any other deformation $F'$ (or unfolding $\tilde{F}'$ respectively) is $\mathcal{A}_{\omega,H}$-equivalent to a deformation induced from $F$, i.e. there exists a relative diffeomorphism $\phi:(\mathbb{C}^{n+1}\times \mathbb{C}^{k'},0)\rightarrow (\mathbb{C}^{n+1},0)$, $\phi(.;0)=.$, preserving both $H$ and $\omega$, a relative diffeomorphism $\psi : (\mathbb{C}\times \mathbb{C}^k,0)\rightarrow (\mathbb{C},0)$, $\psi(.;0)=.$ and a map germ $g:(\mathbb{C}^{k'},0)\rightarrow (\mathbb{C}^{k},0)$ such that:  
\[\psi(F(\phi(x,y;\lambda');g(\lambda'))=F'(x,y;\mathbb{\lambda'}).\]

Let us consider now the corresponding infinitesimal isochore deformations. The space of non-trivial isochore deformations of the germ $(f,H)$ is, as is easily seen, the space: 
\[\tilde{I}^1_{f,H}=\frac{\mathcal{O}_{n+1}}{\{L_vf+k(f)/L_v\omega=0, \hspace{0.2cm} v|_{H}\in TH\} }.\]
This is a $\mathbb{C}\{f\}$-module which can be viewed as the quotient of the ``isochore Jacobian module'' of the boundary singularity $(f,H)$\footnote{in analogy with the isochore Jacobian module of an ordinary singularity \cite{G1}, it is the space of non-trivial infinitesimal deformations with respect to (right) $\mathcal{R}_{\omega,H}$-equivalence.}:
\[I^1_{f,H}=\frac{\mathcal{O}_{n+1}}{\{L_vf/L_v\omega=0, \hspace{0.2cm} v|_{H}\in TH\} }\]
by the submodule generated by the class of the constant function $1$. The latter module is in turn isomorphic to the relative Brieskorn module $H''_{f,H}$ of the boundary singularity, the isomorphism given by multiplication with the volume form $\omega$, and consequently it is free of rank $\mu_{f,H}$. Thus, a necessary condition for a deformation $F$ of $(f,H)$ to be isochore versal is that the classes of the velocities $\partial_{\lambda_i} F:=\frac{\partial F}{\partial \lambda_i}|_{\lambda=0}$ along with the class of $1$, span the isochore Jacobian module $I^1_{f,H}$ over $\mathbb{C}\{f\}$. The following theorem is an analog of the Garay-Mather theorem \cite{G2} and says that this condition is also sufficient:
\begin{thm}
\label{t7}
A deformation $F:(\mathbb{C}^{n+1}\times \mathbb{C}^k,0)\rightarrow (\mathbb{C},0)$ of a boundary singularity $(f,H)$ is isochore versal if it is infinitesimally isochore versal, i.e.
\begin{equation}
\label{iv}
I^1_{f,H}=\text{span}_{\mathbb{C}\{f\}}\{1,\partial_{\lambda_1}F,...,\partial_{\lambda_k}F\}\Leftrightarrow H''_{f,H}=\text{span}_{\mathbb{C}\{f\}}\{\omega,\partial_{\lambda_1}F\omega,...,\partial_{\lambda_k}F\omega\}
\end{equation}
\end{thm} 
Following \cite{G1} we may prove this theorem as follows: first we show that any 1-parameter deformation $G$ of an infinitesimally versal deformation $F$ is isochore trivial (we call $F$ isochore rigid in analogy with the ordinary case). Then we conclude by using J. Martinet's trick, according to which any $k$-parameter deformation can be considered as a ``sum'' of 1-parameter deformations.  The isochore rigidity in turn can be interpreted cohomologically in terms of a parametric version of the relative Brieskorn module which we present below.

\subsubsection{The Parametric Relative Brieskorn Module and Isochore Rigidity}

Let $\Omega^{\bullet}_{n+1+k}$ denote the complex of germs of holomorphic forms at the origin of $\mathbb{C}^{n+1}\times \mathbb{C}^k$ and let $\Omega^{\bullet}_{n+1+k}(H)$ denote the subcomplex of forms vansihing on $H$. In a coordinate system $(x,y_1,...,y_n;\lambda_1,...,\lambda_k)$ for which $H=\{x=0\}$ we have explicitly $\Omega^{\bullet}_{n+1+k}(H)=x\Omega^{\bullet}_{n+1+k}+dx\wedge \Omega^{\bullet-1}_{n+1+k}$. In analogy with the case of the germ $(f,H)$ we may define a relative de Rham cohomology for the map $\tilde{F}$ (and for the map $\tilde{F}|_H$) as well as the corresponding Brieskorn modules. Here we will only need to consider the parametric version of the relative Brieskorn module $H''_{f,H}$, i.e the $\mathbb{C}\{F,\lambda\}$-module:
\[H''_{F,H}:=\frac{\Omega^{n+1+k}_{n+1+k}}{d\lambda_1\wedge ...\wedge d\lambda_k\wedge dF\wedge d\Omega^{n-1}_{n+1+k}(H)},\]
which plays a crucial role in the proof of the isochore unfolding Theorem \ref{t6}. In the ordinary case \cite{G1}, the finiteness (and freeness) of the parametric Brieskorn module follows from the results of G. M. Greuel \cite{Gr} on the isolated complete intersection singularities. For the boundary case we will only need the following relative part:
\begin{prop}
\label{pr}
The parametric Brieskorn module $H''_{F,H}$ of a deformation $F$ of a boundary singularity $(f,H)$ is finitely generated over $\mathbb{C}\{F,\lambda\}$ and it is of rank $\mu_{f,H}$. Moreover, its restriction on $\mathbb{C}^{n+1}=\{\lambda_1=0,...,\lambda_k=0\}$ is isomorphic to the Brieskorn module $H''_{f,H}$ of $(f,H)$. 
\end{prop} 
\begin{proof}
Since the singularities of $\tilde{F}$ are isolated, the proof of the finitness of the Brieskorn module $H''_{F,H}$ is again a straightforward corollary of the relative analog of the Kiehl-Verdier theorem (c.f. \cite{G0} and references therein). The rank of this module is then equal to the dimension of its fiber for any $(t,\lambda)$ sufficiently close to the origin and in the complement of the discriminant of $\tilde{F}$. By the same reasoning as in Section 2 (a parametric version of the de Rham theorem), this is exactly equal to the dimension of the relative cohomology $H^n(X_t,X'_t;\mathbb{C})$, i.e. equal to $\mu_{f,H}$. The fact the the restriction of $H''_{F,H}$ to $\{\lambda_1=0,...,\lambda_k=0\}$ is isomorphic to $H''_{f,H}$ is obvious from the definition.   
\end{proof}

Consider now a 1-parameter deformation $G_t$ of $F$:
\[G_t:=G:(\mathbb{C}^{n+1}\times \mathbb{C}^k\times \mathbb{C},0)\rightarrow (\mathbb{C},0), \hspace{0.3cm} (x,y;\lambda,t)\mapsto G(x,y;\lambda,t),\]
\[G(x,y;\lambda,0)=F(x,y;\mathbb{\lambda}).\] 
Then, as is easily seen, $G_t$ is isochore trivial provided that there exists a decomposition:
\begin{equation}
\label{he}
\partial_tG=k(G,\lambda, t)+\sum_{i=1}^kc_i(G,\lambda, t)\partial_{\lambda_i}G+L_vG,
\end{equation}
where $v$ is a relative vector field tangent to the boundary and preserving $\omega$. Multiplying with $\tilde{\omega}=\omega\wedge d\lambda^k\wedge dt$ (where we denote $d\lambda^k=d\lambda_1\wedge ...\wedge d\lambda_k$), the condition of isochore triviality above can be viewed as the condition that the class of the form $\partial_tG\tilde{\omega}$ in the Brieskorn module $H''_{G,H}$ of $G$ (of the unfolding $\tilde{G}$) belongs to the $\mathbb{C}\{G,\lambda,t\}$-module spanned by the classes of form $\tilde{\omega}$ and of the initial velocities  $\partial_{\lambda_i}G\tilde{\omega}$:
\[\partial_tG\tilde{\omega}\in M=\text{span}_{\mathbb{C}\{G,\lambda,t\}}\{\tilde{\omega},\partial_{\lambda_1}G\tilde{\omega},...,\partial_{\lambda_k}G\tilde{\omega}\}.\] 
We will show that if $F$ is infinitesimally isochore versal, then in fact $M=H''_{G,H}$, which implies in turn the existence of a solution of the homological equation (\ref{he}). To prove the assertion, notice that since the Brieskorn module $H''_{G,H}$ is finitely generated, by the above Proposition \ref{pr}, it suffices to show, by Nakayama's lemma, that the image of $M$ by the natural projection:
\[\pi:H''_{G,H}\rightarrow \frac{H''_{G,H}}{\frak{m}H''_{G,H}},\]
coincides with the whole $\mu_{f,H}$-dimensional $\mathbb{C}$-vector space:
\begin{equation}
\label{c}
\pi(M)=\frac{H''_{G,H}}{\frak{m}H''_{G,H}}.
\end{equation} 
Here $\frak{m}$ is the maximal ideal in $\mathcal{O}_{\mathbb{C}\times \mathbb{C}^k\times \mathbb{C},0}$. But according to Proposition \ref{pr} again, there is an isomorphism of $\mu_{f,H}$-dimensional vector spaces:
\[\frac{H''_{G,H}}{\frak{m}H''_{G,H}}\cong \frac{H''_{f,H}}{fH''_{f,H}}.\]
Thus the condition (\ref{c}) above reduces to the condition:
\begin{equation}
\label{c1}
\pi(M)=\frac{\text{span}_{\mathbb{C}\{f\}}\{\omega,\partial_{\lambda_1}F\omega,...,\partial_{\lambda_k}F\omega\}}{fH''_{f,H}}=\frac{H''_{f,H}}{fH''_{f,H}},
\end{equation} 
which is in turn equivalent, by Nakayama's lemma, to the assumption (\ref{iv}) of infinitesimal isochore versality of $F$. Thus we have proved:
\begin{prop}
\label{pr1}
An infinitesimally isochore versal deformation of a boundary singularity is isochore rigid.
\end{prop}

\subsubsection{Proof of the Isochore Versal Deformation Theorem and Corollaries}
\begin{proof}[Proof of Theorem \ref{t7}]
It goes exactly as in \cite{G1} and relies in a standard trick of J. Martinet  which can be adapted with no problem to the boundary case: let $F$ be a deformation of $(f,H)$, $f=F(.,0)$ and $G$ another deformation of $(f,H)$. Define the sum $F\oplus G$ by:
\[F\oplus G(x,y;\lambda,\lambda')=F(x,y;\lambda)+G(x,y;\lambda')-f(x,y).\]
The restriction of $F\oplus G$ on $\lambda=0$ is equal to $G$ and thus, in order to show that $G$ is isochore equivalent to a deformation induced by $F$, it suffices to show that the deformation $F\oplus G$ is an isochore trivial deformation of $F$.  This can be shown inductively as follows: denote by $F_j$ the restriction of $F\oplus G$ to $\{\lambda_j=...=\lambda_k=0\}$. Then $F_1=F$ and $F_k=F\oplus G$. It follows from Proposition \ref{pr1} that for each $j$ the deformation $F_{j-1}$ is isochore rigid and thus $F_j$ is an isochore trivial deformation of $F_{j-1}$. We conclude by induction that $F_k$ is an isochore trivial deformation of $F_1$.    
\end{proof}

As an immediate corollary we obtain another proof of the relative isochore Morse lemma \ref{t6}: consider $f_t=f_0+th$, $t\in [0,1]$, a 1-parameter deformation of $f_0$, $f_1=f$, such that $f_t|_H$ has a nondegenerate critical point at the origin for all $t$.  Then for any point $t_0 \in [0,1]$ the germ at $t_0$ of the deformation $f_{t}$ is an isochore trivial deformation of $f_{t_0}$. Indeed, the relative Brieskorn module $H''_{f_t,H}$ is generated by the class of the form $dx\wedge dy^n\wedge dt$  and the claim follows from the isochore deformation theorem. Thus, for any $\epsilon$ sufficiently small, the germ $f_{t_0+\epsilon}$ is isochore equivalent to $f_{t_0}$, and thus $f_0$ is isochore equivalent to $f_1$ as well.  

As another immediate corollary we obtain also a relative version of a theorem of Y. Colin de Verdi\`ere \cite{C}, i.e. that a versal deformation of a quasihomogeneous boundary singularity is isochore versal. Indeed, in this case there is an isomorphism (c.f. Example \ref{ex}):
\[\frac{H''_{f,H}}{fH''_{f,H}}\cong \mathcal{Q}_{f,H}\]
and thus the classes of 1 with the initial velocities of the deformation generate the isochore Jacobian module $I^1_{f,H}$.

\section*{Acknowledgements}
The author is grateful to Mauricio D. Garay for several useful discussions and his attention on the problem. The author would also like to thank the referee for his accurate suggestions.\\
 ``This research has been supported (in part) by EU Marie-Curie IRSES Brazilian-European partnership in Dynamical Systems (FP7-PEOPLE-2012-IRSES 318999 BREUDS)''.

\end{document}